\documentclass{article}
\usepackage[top=2cm,left=2.5cm,right=3cm,right=2cm]{geometry}
\usepackage[utf8]{inputenc}
\usepackage[english]{babel}
\usepackage{amsmath}
\usepackage{mathtools}
\usepackage{amsfonts}
\usepackage{amssymb}
\usepackage{amsthm}
\usepackage{bbold}
\usepackage{ifthen}
\usepackage{xcolor}
\usepackage[makeroom]{cancel}
\usepackage{graphicx}
\usepackage[square,numbers,sort&compress]{natbib}
\usepackage{tikz}

\usepackage{mathrsfs}

\usepackage{hyperref}
\hypersetup{
        colorlinks=true,
        linkcolor=black,
        citecolor=black,
        urlcolor=black,
}
\usepackage[aligntableaux=center]{ytableau}
\usepackage{float}
\usepackage{multirow}
\usepackage{accents}
\usepackage{tensor}
\usepackage{prodint}

\newcommand{\tr}[2]{\mathrm{tr}^{\vphantom{+}}_{#1#2}}
\newcommand{\ctr}[2]{\mathrm{tr}^{+}_{#1#2}}
\newcommand{\sfa}{\mathsf{a}}
\newcommand{\hsfa}{\mathsf{a}^{\prime}}
\newcommand{\sfb}{\mathsf{b}}
\newcommand{\hsfb}{\mathsf{b}^{\prime}}
\newcommand{\scrA}{\mathscr{A}}
\newcommand{\scrC}{\mathscr{C}}
\newcommand{\scrL}{\mathscr{L}}
\newcommand{\scrR}{\mathscr{R}}
\newcommand{\scrP}{\mathscr{P}}

\newcommand{\Sn}[1]{\mathfrak{S}_{#1}}
\newcommand{\xnone}{\phantom{\times}}

\newcommand{\bref}[1]{{\bf \ref{#1}}}

\newcommand\xfootnote[1]{%
  \begingroup
  \renewcommand\thefootnote{}\footnote{#1}%
  \addtocounter{footnote}{-1}%
  \endgroup
}

\newtheorem{proposition}{Proposition}[section]
\newtheorem{lemma}[proposition]{Lemma}
\newtheorem{corollary}[proposition]{Corollary}
\newtheorem{theorem}[proposition]{Theorem}

\numberwithin{equation}{section}

\begin{document}
\ 
\vspace{0.5 cm}
\begin{center}
{\Large \textbf{Traceless projection of mixed tensor products,\\and walled Brauer algebras
}}

{\small
\vspace{.55cm} 
{\large \textbf{Y.~O. Goncharov$^{1,2}$}\xfootnote{\texttt{yegor.goncharov@idpoisson.fr}}}

\vskip .35cm $^{1}$ ESIEE Paris, Universit\'{e} Gustave Eiffel,\\
28 avenue Andr\'{e}-Marie Amp\`{e}re, 77420 Champs-sur-Marne, France

\vskip .35cm $^{2}$ Institut Denis Poisson,
Universit\'e de Tours, Universit\'e d’Orl\'eans, CNRS,\\
Parc de Grandmont, 37200 Tours, France

}

\end{center}

\begin{abstract}\noindent
Along similar lines as in \cite{traceless}, we describe a self-contained procedure for constructing the traceless projection of mixed tensor products (built out of a finite-dimensional complex vector space and its dual). The  construction relies on the Schur-Weyl duality for the general linear group and regards rational representations thereof. By identifying the traceless subspace as a particular rational representation, the traceless projector which commutes with the group action can be understood as a uniquely defined idempotent in the centraliser algebra. We also identify and construct the analogue of the traceless projector in the walled Brauer algebras when the latter are semisimple. Among possible applications of the traceless projector, we show how the result applies to mixed tensor products built out of a finite-dimensional hermitian space and its complex conjugate.
\end{abstract}

\section{Preliminaries and notations}

\subsection{Introduction}\label{sec:intro}

    \paragraph{Traceless projection of mixed tensors.} Let $V$ be a finite-dimensional complex vector space with $\dim V = N$, and let $V^{*}$ denote the dual space. For any pair of positive integers $m,n$ the mixed tensor product
    \begin{equation}\label{eq:intro_mixed_tensor}
        \underbrace{V\otimes \ldots \otimes V}_{m}\otimes \underbrace{V^{*} \otimes \ldots \otimes V^{*}}_{n}
    \end{equation}
    admits the uniquely defined {\it traceless subspace} whose elements vanish upon applying the canonical contraction to pairs $V$ and $V^{*}$ at any given positions in \eqref{eq:intro_mixed_tensor}. Projection onto the traceless subspace is referred to as traceless projection and constitutes the main focus of the present work.
    \vskip 2pt

    Like any projection in general, projection onto the traceless subspace is defined modulo the choice of a complementary subspace which constitutes the kernel of the corresponding traceless projector. The natural choice of the traceless projector is provided by the invariance of the traceless subspace under the diagonal action of the full linear group $GL(N)$, when an invertible transformation of $V$ is accompanied by the dual (contragredient) transformation of $V^{*}$. In this way, \eqref{eq:intro_mixed_tensor} is a rational representation of $GL(N)$ which is known to be completely reducible \cite{Schur_dissertation}, so the traceless subspace admits the unique $GL(N)$-invariant complement described in Section \bref{sec:traceless_projection}.
    \vskip 2pt

    \paragraph{Schur-Weyl duality.} The group-theoretic background of the problem evokes parallels with \cite{traceless} where the traceless projection of the tensor product $V^{\otimes n}$ was constructed in the case of $V$ equipped with a non-degenerate scalar product. The same key idea applies when $V$ and $V^{*}$ carry no additional structures: along with the subalgebra of transformations $L_{m,n}(N)$ generated by the action of $GL(N)$ in \eqref{eq:intro_mixed_tensor} one also considers its centraliser algebra $C_{m,n}(N)$ constituted by all linear transformations of \eqref{eq:intro_mixed_tensor} which commute with the group action.
    \vskip 2pt

    The algebra $L_{m,n}(N)$ is semisimple, which is an equivalent formulation of the complete reducibility of \eqref{eq:intro_mixed_tensor} under the action of $GL(N)$. As a result, semisimplicity of $C_{m,n}(N)$ follows, as well as that both algebras of endomorphisms are mutual centralisers. These facts underlie the remarkable interplay between the representation theory of $GL(N)$ on one hand and of its centraliser algebra on the other, commonly referred to as {\it Schur-Weyl duality}. In our analysis we exploit the following general features of Schur-Weyl dualities (see, {\it e.g.}, \cite{Stevens_SWd}): let $\mathfrak{A}$ be a semisimple (associative) algebra of linear transformations of a finite-dimensional complex vector space $\mathcal{V}$, and let $\mathfrak{A}^{\prime}$ denote the centraliser algebra, then 
    \begin{itemize}
        \item[1)] The algebra $\mathfrak{A}^{\prime}$ is semisimple and $(\mathfrak{A}^{\prime})^{\prime} = \mathfrak{A}$.
        \item[2)] Equivalence classes of irreducible representations of $\mathfrak{A}$ are in one-to-one correspondence with those of $\mathfrak{A}^{\prime}$. 
        \item[3)] An irreducible representation of $\mathfrak{A}$ (respectively, of $\mathfrak{A}^{\prime}$) occurs in $\mathcal{V}$ with the multiplicity equal to the dimension of the corresponding irreducible representation of $\mathfrak{A}^{\prime}$ (respectively, of $\mathfrak{A}$).
        \item[4)] Irreducible representations of $\mathfrak{A}$ in $\mathcal{V}$ result from applying primitive idempotents of $\mathfrak{A}^{\prime}$.  
    \end{itemize}
    For a detailed review of Schur-Weyl dualities for the classical Lie groups see \cite{GoodmanWallach1998}.
    \vskip 2pt
    
    The primer example of a Schur-Weyl duality regards the action of $GL(N)$ in the tensor product $V^{\otimes n}$. In this case, I.~Schur proved that the centraliser algebra is generated by the action of the enveloping algebra of the permutation group $\mathbb{C}\Sn{n}$ \cite{Schur_dissertation}, while H. Weyl proved that primitive idempotents in $\mathbb{C}\Sn{n}$ project $V^{\otimes n}$ onto irreducible representations of $GL(N)$ \cite{Weyl_1929} (see \cite{Weyl,GoodmanWallach1998} for details). In particular, a possibility to single out irreducible representations of $GL(N)$ in \eqref{eq:intro_mixed_tensor} consists in applying Young symmetrisers.
    \vskip 2pt
    
    In the case when $GL(N)$ acts in the mixed tensor product \eqref{eq:intro_mixed_tensor}, the centraliser algebra $C_{m,n}(N)$ was described in \cite{Koike_89} by means of a generating set of endomorphisms, where a similarity to Brauer algebras \cite{Brauer} was also mentioned. A convenient diagrammatic realisation of $C_{m,n}(N)$ in terms of a particular subalgebra $B_{m,n}(N) \subset B_{m+n}(N)$ in the Brauer algebra, together with its action in \eqref{eq:intro_mixed_tensor}, was given in \cite{BCHLLS} (see Section \bref{sec:wB_algebra}). The latter subalgebras (as well as the Brauer algebras) can be defined irrespectively of their action in tensor-product spaces, in which case one has $B_{m,n}(\delta)$ for any integers $m,n\geqslant 1$ and $\delta \in \mathbb{C}$, referred to as walled Brauer algebras. The latter were independently introduced in \cite{Turaev_1990}.
    \vskip 2pt

    By considering \eqref{eq:intro_mixed_tensor} as a representation of $GL(N)\times GL(N)$, where the left (respectively, right) copy of the group acts independently in the $m$ factors $V$ (respectively, $n$ factors $V^{*}$), one arrives at another example of a Schur-Weyl duality. The corresponding centraliser algebra $S_{m,n}(N)$ is generated by independent permutations of the factors $V$ and $V^{*}$ in \eqref{eq:intro_mixed_tensor}, and thus results from the action of $\mathbb{C}[\Sn{m}\times \Sn{n}]$. 
    \vskip 2pt
    
    The inclusion relations $GL(N) \subset GL(N)\times GL(N)$ (the diagonal subgroup), $\mathbb{C}[\Sn{m}\times \Sn{n}] \subset B_{m,n}(N)$ and finally $S_{m,n}(N) \subset C_{m,n}(N)$ can be summarised via the following {\it see-saw diagram}, reminiscent of those for dual pairs of Lie groups \cite{Kudla_seesaw} which arise in the context of Howe duality (for review, see \cite{Prasad_lecture_1993} and references therein):
    \begin{center}
        \begin{tikzpicture}[baseline=(current bounding box.center)]
            \def \dx {1};
            \def \dy {1.5};
            \node at (0,0) {$B_{m,n}(N)$};
            \node at (2.6*\dx,0) {$C_{m,n}(N)$};
            \node at (2.6*\dx,-\dy) {$S_{m,n}(N)$};
            \node at (8*\dx,0) {$GL(N)\times GL(N)$};
            \node at (7.2*\dx,-\dy) {$GL(N)$};
            \node at (1.3*\dx,0) {$\longrightarrow$};
            \node at (1.36*\dx,0) {$\longrightarrow$};
            \node at (1.3*\dx,-\dy) {$\longrightarrow$};
            \node at (1.36*\dx,-\dy) {$\longrightarrow$};
            \node at (0,-\dy) {$\mathbb{C}[\Sn{m}\times \Sn{n}]$};
            \draw[line width=0.5pt] (0,-0.2*\dy) -- (0,-0.8*\dy);
            \draw[line width=0.5pt] (2.6*\dx,-0.2*\dy) -- (2.6*\dx,-0.8*\dy);
            \draw[line width=0.5pt] (7.2*\dx,-0.2*\dy) -- (7.2*\dx,-0.8*\dy);
            \draw[line width=0.5pt] (3.5*\dx,-0.1*\dy) -- (6.5*\dx,-0.9*\dy);
            \draw[line width=0.5pt] (3.5*\dx,-0.9*\dy) -- (6.5*\dx,-0.1*\dy);
        \end{tikzpicture}
    \end{center}
    Here vertical lines represent inclusions, arrows denote surjective homomorphisms of algebras and diagonal lines mark Schur-Weyl-dual pairs of algebras. The interplay between the two Schur-Weyl dualities allows us to derive the branching rules for the centraliser algebras (in the middle column of the above diagram) from those for the Lie groups (in the right column), see Proposition \bref{prop:multiplicity_SsubB}.
    \vskip 2pt

    Let us note that in physics, Howe duality (and see-saw dual pairs in particular) is known through oscillator representations of the classical Lie groups (for review, see \cite{JouBaMrktMoj_2020} and references therein).
    
    \paragraph{Traceless projector.}  In all the aforementioned Schur-Weyl dualities, projection onto a particular group representation in the tensor product consists in applying a suitable idempotent of the centraliser algebra. Thus, the $GL(N)$-invariant projection of \eqref{eq:intro_mixed_tensor} onto the traceless subspace consists in constructing a particular (uniquely-defined) idempotent
    \begin{equation}\label{eq:intro_traceless_projector}
        \scrP_{m,n} \in C_{m,n}(N).
    \end{equation}

    Up to this point, we admit that the main lines of the construction of the traceless projector may be familiar to specialists in the representation theory of $GL(N)$ and $B_{m,n}(N)$. Furthermore, one may think of a brute-force way of constructing the traceless projector by summing up appropriate primitive idempotents in $B_{m,n}(N)$. However, we choose to renounce going this way, and consider another possibility due to a number of substantial reasons. Firstly, because the sought projector is a central element of $C_{m,n}(N)$ (see Proposition \bref{prop:P_central}), which makes it much simpler to construct compared to primitive idempotents. Secondly, when $N \leqslant m + n - 1$ the algebra $B_{m,n}(N)$ is non-semisimple \cite[Theorem 6.3]{CDDM}, so addressing the representation theory of the centraliser algebra $C_{m,n}(N)$ by means of $B_{m,n}(N)$ leads to unnecessary complications. Finally, from the point of view of numerical applications of the traceless projector, expressing it as a sum of particular primitive idempotents in $C_{m,n}(N)$ can be hardly viewed as an optimal solution of the problem.
    \vskip 2pt
    
    Construction of the traceless projector \eqref{eq:intro_traceless_projector} goes along similar lines as in \cite{traceless}: one identifies a suitable element $\scrA_{m,n} \in C_{m,n}(N)$ which is diagonalisable in \eqref{eq:intro_mixed_tensor} and whose kernel is exactly the traceless subspace (see Lemma \bref{lem:operator_A}). As a result, given $\mathrm{spec}(\scrA_{m,n})$ (the set of eigenvalues of $\scrA_{m,n}$) the traceless projector is the projector onto the kernel of $\scrA_{m,n}$:
    \begin{equation}\label{eq:intro_traceless_projector_factorised}
        \scrP_{m,n} = \prod_{a \in \mathrm{spec}(\scrA_{m,n})\backslash\{0\}} \left(1 - \dfrac{1}{a} \scrA_{m,n}\right).
    \end{equation}

    The interplay between the two Schur-Weyl dualities presented in the above see-saw diagram allows us to determine $\mathrm{spec}(\scrA_{m,n})$ completely, for any given integers $m,n \geqslant 1$ and $N \geqslant 1$ (see Theorem \bref{thm:eigenvalues}). The result essentially follows from the semisimplicity of the related algebras of transformations, with particular details about irreducible representations of the latter derived from the basic knowledge in the representation theory of the symmetric group on one hand, and of the general linear group on the other. Bypassing the usage of the representation theory of the walled Brauer algebra $B_{m,n}(N)$ allows us to construct $\mathrm{spec}(\scrA_{m,n})$ uniformly for all $N \geqslant 1$, which constitutes the main advantage of the present analysis compared to \cite{traceless}.
    \vskip 2pt
    
    \paragraph{Factorised form of the traceless projector.} The factorised form  of the traceless projector \eqref{eq:intro_traceless_projector_factorised} is a feature of the proposed construction due to the possibility of expressing the traceless subspace as a kernel of a single operator $\scrA_{m,n}$. While the particular choice of the latter is important within the present formalism, it is quite likely that other choices of $\scrA_{m,n}$ are possible. Analysis of this question lies beyond the scope of the present work. 
    \vskip 2pt
    
    Let us outline a number of useful features of the factorised form \eqref{eq:intro_traceless_projector_factorised}. First of all, $\scrA_{m,n}$ commutes with the action of $\mathbb{C}[\Sn{m}\times \Sn{n}]$ in \eqref{eq:intro_mixed_tensor}, which manifests the fact that \eqref{eq:intro_traceless_projector} preserves permutation symmetries of contravariant and covariant indices of a tensor. Moreover, when sequentially applying the factors of \eqref{eq:intro_traceless_projector_factorised}, permutation symmetries of a tensor are preserved at each step.
    \vskip 2pt
    
    The factorised form \eqref{eq:intro_traceless_projector_factorised} provides considerable flexibility to the construction. First of all, the right-hand side of \eqref{eq:intro_traceless_projector_factorised} expresses the same traceless projector upon extending $\mathrm{spec}(\scrA_{m,n})$ to any finite subset of $\mathbb{C}$. We give a particular extension $\widetilde{\mathrm{spec}}(\scrA_{m,n})$ which arises from the representation theory of $B_{m,n}(N)$, and provide a sufficient condition for $\widetilde{\mathrm{spec}}(\scrA_{m,n}) =  \mathrm{spec}(\scrA_{m,n})$ (see Proposition \bref{prop:spec_A_alternative}). In particular, this happens for $N \geqslant m + n - 1$ when the algebra $B_{m,n}(N)$ is semisimple \cite{CDDM}. The necessary condition for the two sets to coincide does not manifest itself among the simplest examples, and rests unknown to the author.
    \vskip 2pt

    Flexibility of the factorised form of the traceless projector is also due to the possibility of reducing the number of factors in \eqref{eq:intro_traceless_projector_factorised} when it is applied to a tensor in an irreducible representation of $GL(N)\times GL(N)$. In particular, one can think of tensors which result from Young symmetrisations of contravariant and covariant indices. Depending on the equivalence class of a $GL(N)\times GL(N)$-representation, we describe the smallest subset of $\mathrm{spec}(\scrA_{m,n})$ to be used in \eqref{eq:intro_traceless_projector_factorised}, sufficient for constructing the traceless projection. This result is summarised in Theorem \bref{thm:eigenvalues_restricted}.

    \paragraph{Splitting idempotent.}  The traceless projector \eqref{eq:intro_traceless_projector} can be equivalently identified as a central idempotent in $C_{m,n}(N)$ which splits the centraliser algebra into two complementary ideals, one of which is the annihilator ideal of the traceless subspace. By denoting the latter ideal $\mathcal{J} \subseteq C_{m,n}(N)$ one has the following exact sequence of algebras:
    \begin{equation}\label{eq:intro_exact_sequence_C}
        0 \xrightarrow{\phantom{m}} \mathcal{J} \xrightarrow{\phantom{m}} C_{m,n}(N) \xrightarrow{\phantom{m}} C_{m,n}(N)\slash \mathcal{J} \xrightarrow{\phantom{m}} 0.
    \end{equation}
    
    The above sequence is split exact. Indeed, because the algebra $C_{m,n}(N)$ is semisimple ({\it i.e.} isomorphic to a direct sum of full matrix algebras),  $\mathcal{J}$ can be complemented by another ideal $\mathcal{I} \subset C_{m,n}(N)$ such that
    \begin{equation}\label{eq:intro_decomposition_C}
        C_{m,n}(N) = \mathcal{I} \oplus \mathcal{J}\quad \text{(direct sum of algebras).}
    \end{equation}
    The two ideals result from multiplying the elements of $C_{m,n}(N)$ by the uniquely-defined central idempotents which form a decomposition of unity in $C_{m,n}(N)$. We refer to the central idempotent which projects $C_{m,n}(N)$ onto $\mathcal{I}$ as {\it splitting idempotent} of \eqref{eq:intro_exact_sequence_C} and show that it coincides with the traceless projector \eqref{eq:intro_traceless_projector}, so that
    \begin{equation}\label{eq:intro_split_ideal}
        \mathcal{I} = \scrP_{m,n} C_{m,n}(N).
    \end{equation}
    This result is formulated in Theorem \bref{thm:splitting_idempotent_C}.
    \vskip 2pt
    
    The relation \eqref{eq:intro_split_ideal} shows that $\mathcal{I} \cong S_{m,n}(N)$, while $S_{m,n}(N) \cong \mathbb{C}[\Sn{m}\times\Sn{n}]$ if and only if $N \geqslant m + n$ (see Proposition \bref{prop:image_Sn}). The `if' part of the latter assertion follows from the isomorphism between $C_{m,n}(N)$ and $B_{m,n}(N)$ \cite[Theorem 5.8]{BCHLLS}, in view of the well-known fact \eqref{eq:wB_direct_sum_SJ} for the walled Brauer algebras. At the same time, the author is unaware of the `only if' implication in the literature, so a simple proof thereof based on the analysis of the traceless projection is proposed. In a similar fashion, we give a proof that the homomorphism from $B_{m,n}(N)$ onto $C_{m,n}(N)$ is injective only if $N \geqslant m + n$.
    \vskip 2pt
        
    Viewing the traceless projector as a splitting idempotent (along similar lines as discussed in \cite{KMP_central_idempotents}) we define the analogue of the traceless projector $P_{m,n}\in B_{m,n}(\delta)$ in the walled Brauer algebra for all but a finite set of $\delta \in \mathbb{C}$, when $B_{m,n}(\delta)$ is semisimple \cite[Theorem 6.3]{CDDM} (see Theorem \bref{thm:wB_splitting_idempotent}). The role of the annihilator ideal of the traceless subspace is played by the ideal $J \subset B_{m,n}(\delta)$ spanned by the diagrams with at least one arc (see Section \bref{sec:wB_algebra}). In the case when $\delta = N \in \mathbb{N}$ such that $N \geqslant m + n - 1$, the action of $P_{m,n}$ in \eqref{eq:intro_mixed_tensor} coincides with that of the traceless projector \eqref{eq:intro_traceless_projector}. Let us note that constructing the analogue of $P_{m,n}$ in the Brauer algebra was the starting point of the construction of the traceless projector in \cite{traceless} (when the vector space is equipped with a non-degenerate scalar product): by doing so one obtains a self-contained factorised formula for the traceless projector, with a leftover uncertainty about whether certain factors can be omitted when the dimension of the vector space is small.
    
    \paragraph{Organisation of the paper.} Further presentation goes as follows. Two more subsections complete the introduction: in Section \bref{sec:intro_applications} we outline possible domains of application of the traceless projector for mixed tensor products, and in Section \bref{sec:intro_notations} we fix notations for Young diagrams and recall some basic features of rational characters of $GL(N)$ which are utilised in the sequel.
    \vskip 2pt
    
    In Section \bref{sec:traceless_projection} we identify the traceless subspace and its complement in the mixed tensor product \eqref{eq:intro_mixed_tensor}, introduce the endomorphism $\scrA_{m,n}$ and give a self-contained procedure for computing its eigenvalues. In Section \bref{sec:traceless_rep_theory} we explain the construction by providing the relevant details about group actions, centraliser algebras and their representations. Branching rules from $C_{m,n}(N)$ to $S_{m,n}(N)$, and particular advantages of the factorised form of the traceless projector are discussed in Section \bref{sec:projector_further}.
    \vskip 2pt

    In Section \bref{sec:splitting_idempotent_C} we identify the traceless projector as a particular splitting idempotent. The analogue of the traceless projector for walled Brauer algebras is constructed in Section \bref{sec:wB_algebra}.

    \subsection{Possible applications and developments.}\label{sec:intro_applications}

    Unlike the usual situation within applications of tensors in physics (say, in particle physics, gravity and elasticity), the notion of trace considered in the present work is not related to a metric. Another feature of the construction is that it essentially relies on complex vector spaces. As a result, possible applications of the traceless projector \eqref{eq:intro_traceless_projector} may be not apparent from the first glance and thus merit a discussion.  

    \paragraph{Real vector spaces.} Despite the proposed construction regards complex vector spaces, the resulting traceless projector \eqref{eq:intro_traceless_projector} equally applies in the case of finite-dimensional vector fields over $\mathbb{Q}$. This follows directly from the fact that $\mathrm{spec}(\scrA_{m,n})\backslash\{0\} \subset \mathbb{N}$ (see Theorem \bref{thm:eigenvalues} and recall \eqref{eq:intro_traceless_projector_factorised}). As a result, the traceless projector \eqref{eq:intro_traceless_projector_factorised} applies in the case of vector spaces over $\mathbb{R}$. More generally, one can consider any field $K$ such that $\mathbb{Q} \subseteq K$.
    
    \paragraph{Unitary groups $U(p,q)$.} Let us show how the traceless projector \eqref{eq:intro_traceless_projector} can be applied when the group $GL(N)$ is replaced by any of its real forms $U(p,q)$ with $p+q = N$, in which case for any $m,n \geqslant 1$ one considers mixed tensor products
    \begin{equation}\label{eq:intro_mixed_conjugate}
       \underbrace{V\otimes \ldots \otimes V}_{m}\otimes \underbrace{\bar{V} \otimes \ldots \otimes \bar{V}}_{n}
    \end{equation}
    built out of a $N$-dimensional complex vector space $V$ and its complex conjugate $\bar{V}$. The space $V$ is canonically equipped with a non-degenerate hermitian form $(\cdot,\cdot)$ which provides a $U(p,q)$-invariant isomorphism between the conjugate $\bar{V}$ and the dual $V^{*}$ spaces:
    \begin{equation}\label{eq:intro_iso_VVbar}
    \begin{array}{rccc}
        h : & \bar{V} & \rightarrow & V^{*} \\
        \hfill & v & \mapsto & (v,\cdot\,)
    \end{array}
    \end{equation} 
    At the same time, no natural isomorphism between $V$ and $\bar{V}$ (and thus between $V$ and $V^{*}$) is around. 
    \vskip 2pt
    
    The isomorphism \eqref{eq:intro_iso_VVbar} allows one to apply the traceless projector \eqref{eq:intro_traceless_projector} to the mixed tensor product \eqref{eq:intro_mixed_conjugate} as follows (see Section \bref{sec:traceless_projection} for details). Given a basis $\{e_{i}\}$ in $V$, consider the non-degenerate hermitian matrix $g_{ij} = (e_{i},e_{j})$ (such that $g_{ij} = \bar{g}_{ji}$, where $\bar{z}$ denotes the complex conjugate of $z\in\mathbb{C}$) and its inverse $g^{ij}$ such that $g^{ik}g_{kj} = \delta^{i}_{j}$. Unitary transformations preserve the hermitian form, so any $S \in U(p,q)$ is represented by a non-degenerate matrix subject to the following condition:
    \begin{equation}\label{eq:unitary_matrix}
        g_{kl}S^{k}{}_{i} S^{l}{}_{j} = g_{ij}.
    \end{equation}

    Consider the basis $\{\bar{e}_{i}\}$ in $\bar{V}$ such that 
    \begin{equation}
        S(e_{i}) = e_{j} \,S^{j}{}_{i}\quad \Leftrightarrow \quad S(\bar{e}_{i}) = \bar{e}_{j} \,\bar{S}^{j}{}_{i},
    \end{equation}
    and consider the following change of basis in $\bar{V}$:
    \begin{equation}\label{eq:intro_basis_dual}
        e^{i} = g^{ij}\bar{e}_{j} \quad \Leftrightarrow \quad \bar{e}_{i} = g_{ij} e^{j}.
    \end{equation}
    As a matter of a routine exercise one checks that \eqref{eq:unitary_matrix} assures the contragredient transformation law for the new basis, see \eqref{eq:action_basis}:
    \begin{equation}
        S(e^{i}) = (S^{-1})^{i}{}_{j}\, e^{j},
    \end{equation}
    so that the isomorphism \eqref{eq:intro_iso_VVbar} is realised manfestly. As a result, one reads off the action of \eqref{eq:intro_traceless_projector} in \eqref{eq:intro_mixed_conjugate} via the following change of basis in \eqref{eq:intro_mixed_conjugate} induced by \eqref{eq:intro_basis_dual}:
    \begin{equation}
        e_{i_{1}} \otimes \ldots \otimes e_{i_{m}}\otimes \bar{e}_{j_{1}} \otimes \ldots \otimes \bar{e}_{j_{n}} = g_{j_{1}k_{1}}\ldots g_{j_{n}k_{n}} \, e_{i_{1}} \otimes \ldots \otimes e_{i_{m}}\otimes e^{k_{1}} \otimes \ldots \otimes e^{k_{n}},
    \end{equation}
    where the action on the basis vectors $e_{i_{1}} \otimes \ldots \otimes e_{i_{m}}\otimes e^{k_{1}} \otimes \ldots \otimes e^{k_{n}}$ is described in Section \bref{sec:traceless_projection}.
    \vskip 2pt
    
    Another convenient possibility regards a minor modification of the rule given in Section \bref{sec:wB_algebra} for the action of the walled Brauer algebra on the components of a tensor
    \begin{equation}
        T = t^{i_{1}\ldots i_{m} | j_{1} \ldots j_{n}}\, e_{i_{1}} \otimes \ldots \otimes e_{i_{m}}\otimes \bar{e}_{j_{1}} \otimes \ldots \otimes \bar{e}_{j_{n}}.
    \end{equation}
    Namely, that an arc with endpoints $1 \leqslant \sfa \leqslant m$ and $1 \leqslant \hsfb \leqslant n$ in the upper (respectively, lower) row of a walled diagram encodes contraction with $g_{i_{\sfa}j_{\hsfb}}$ (respectively, multiplication by $\bar{g}^{i_{\sfa}j_{\hsfb}}$).

    \paragraph{Non-semisimple subgroups of $GL(N)$.}  The general situation where mixed tensor products arise regards vector spaces endowed with a degenerate symmetric metric, which does not provide an isomorphism between the vector space and its dual. This is the case, for example, when the $N$-dimensional vector space $V$ realises the faithful indecomposable representation of the Euclidean group $O(N-1)\ltimes \mathbb{C}^{N-1} \subset GL(N)$. Note that the Euclidean group contains the Poincar\'e group as a particular real form, in which case the representation in question arises in the Cartan formulation of gravity (see \cite{FrancoisRavera_CartanGeom} and references therein). For more examples of a kind one has the affine group $GL(N-1)\ltimes \mathbb{C}^{N-1}$, as well as the Galilei and Carroll groups \cite{LevyLeblond_Galilei_63,LevyLeblond_65,Gupta_Carroll} which arise in the context of the studies of asymptotic symmetries and gravitational waves \cite{Duval_ConfCarrol_2014,Duval_Caroll_grav_waves}, as well as in higher-spin field theories (see for example \cite{CampPek_flatHS_2021,BekPek_2023}).
    \vskip 2pt
    
    The aforementioned groups are non-semisimple, so ``minimal'' faithful representations thereof are indecomposable but not irreducible. As a result, the role of the centraliser algebra, and thus the notion of Schur-Weyl duality, needs to be clarified (recall that in the examples of Schur-Weyl dualities listed in Section \bref{sec:intro}, the group action and its centraliser algebra are both semisimple). On the other hand, for the study of tensor products \eqref{eq:intro_mixed_tensor} of faithful indecomposable representations of any of the above groups $H \subset GL(N)$ (for an appropriate integer $N$) one can start with decomposing \eqref{eq:intro_mixed_tensor} into irreducible representations of $GL(N)$, which become reducible upon restriction to the subgroup $H$. In this respect, since the centraliser algebra of the action of $H$ in \eqref{eq:intro_mixed_tensor} contains $C_{m,n}(N)$ as a subalgebra, the traceless projection of any $GL(N)$-invariant subspace of \eqref{eq:intro_mixed_tensor} is $H$-invariant. Further decomposition of the latter into indecomposable summands can be analysed by other methods. Let us mention in this respect that the groups in question result from contractions of simple Lie groups \cite{Inonu-Wigner}.
    \vskip 2pt
    
    \paragraph{Manifolds with an affine connection.} Consider a smooth manifold $M$ (with $\dim M = d$) endowed with an affine connection $\nabla$ and a degenerate metric. Geometries of this type arise, for example, in the covariant description of non-relativistic spacetimes \cite{Kunzle_1972,Duval_Kunzle_newtonian_1977}, as well as in the geometry of paths \cite{VeblenThomas_paths} (where the absence of metric can be understood as if it were trivial). In this case, one has the two canonical tensor fields on $M$ (sections of the bundle $\mathcal{T}M$): the torsion $T$ and the Riemann tensor $R$, both being point-wise mixed tensors of type $(1,2)$ and $(1,3)$ respectively. In order to write down the components of the latter, let $\{x^{i}\}$ (with $i = 1,\ldots,d$) be the set of local coordinates on an open patch of $M$ and denote $\Gamma^{i}{}_{jk}$ the the Christoffel symbols (the components of $\nabla$ in the coordinate basis). Then one has
    \begin{equation}\label{eq:intro_TR}
        T^{i}{}_{jk} = \Gamma^{i}{}_{jk} - \Gamma^{i}{}_{kj} \quad \text{and}\quad R^{i}{}_{j,kl} = \dfrac{\partial \Gamma^{i}{}_{lj}}{\partial x^{k}} - \dfrac{\partial \Gamma^{i}{}_{kj}}{\partial x^{l}} + \Gamma^{i}{}_{kr}\Gamma^{r}{}_{lj} - \Gamma^{i}{}_{lr}\Gamma^{r}{}_{kj}
    \end{equation}
    with $T^{i}{}_{jk} = - T^{i}{}_{kj}$ and $R^{i}{}_{j,kl} = - R^{i}{}_{j,lk}$.
    \vskip 2pt

    Given a point $p\in M$, diffeomorphisms of $M$ which stabilise $p$ induce $GL(d)$-transformations of $\mathcal{T}_{p}M$, so one can think of decomposing the fibers of $\mathcal{T}M$ into irreducible components. In this respect, the traceless projection furnishes a particular direct sum of irreducibles, while further projection onto a particular irreducible component is achieved via (anti-)symmetrisations of tensor indices (in the case of mixed tensors see \cite[Theorem 1.1]{Koike_89}). The traceless projection of tensors with specific permutation symmetries is analysed in Section \bref{sec:projector_further}, where it is shown that the traceless projections of $T$ and $R$ for $d \geqslant 3$ is achieved by applying the operators
    \begin{equation}
        \left(1 - \dfrac{1}{d-1} \scrA_{1,3}\right)\quad \text{and}\quad \left(1 - \dfrac{1}{d+1} \scrA_{1,3}\right)\left(1 - \dfrac{1}{d-1} \scrA_{1,3}\right)\left(1 - \dfrac{1}{d-2} \scrA_{1,3}\right)
    \end{equation}
    respectively (for explanation see \eqref{eq:traseless_projector_torsion} and \eqref{eq:traceless_projector_Riemann}). By applying the above operators to the components \eqref{eq:intro_TR} one reads off the following traceless projections:
    \begin{equation}
        (\scrP_{1,2}T)^{i}{}_{jk} = T^{i}{}_{jk} - \dfrac{1}{d-1} \big(\delta^{i}_{j}\, T^{p}{}_{pk} - \delta^{i}_{k}\, T^{p}{}_{pj}\big)
    \end{equation}
    for the torsion (see \eqref{eq:traceless_projector_21_a}), and
    \begin{equation}
    \def\arraystretch{2}
        \begin{array}{rl}
            (\scrP_{1,3}R)^{i}{}_{j,kl} = &  R^{i}{}_{j,kl} - \dfrac{d-1}{(d+1)(d-2)} \,\delta^{i}_{j}\, R^{p}{}_{p,kl} - \dfrac{d^2 - d - 1}{(d+1)(d-1)(d-2)} \, \big(\delta^{i}_{k} \,R^{p}{}_{j,pl} - \delta^{i}_{l} \,R^{p}{}_{j,pk}\big)\\
            \hfill & + \dfrac{1}{(d+1)(d-2)}\, \big(\delta^{i}_{j}\, R^{p}{}_{k,pl} - \delta^{i}_{j}\, R^{p}{}_{l,pk} + \delta^{i}_{k} \,R^{p}{}_{p,lj} - \delta^{i}_{l} \,R^{p}{}_{p,kj}\big)\\
            \hfill & + \dfrac{1}{(d+1)(d-1)(d-2)}\,\big(\delta^{i}_{k}\,R^{p}{}_{l,pj} - \delta^{i}_{l}\,R^{p}{}_{k,pj}\big)
        \end{array}
    \end{equation}
    for the Riemann tensor (see \eqref{eq:traceless_projector_Riemann}).
    \vskip 2pt
    
    \paragraph{Trace decomposition of mixed tensors.} Among possible developments of the proposed formalism, an interesting related problem consists in constructing the complete set of central idempotents which decompose the mixed tensor-product space as the traceless, doubly-traceless, {\it etc.} subspaces complemented by the trace subspace, along similar lines as in \cite{Helpin}.

\subsection{Notations.}\label{sec:intro_notations}

    \paragraph{Partitions and Young diagrams.} Given a positive integer $s$, a {\it partition} of $s$ is a sequence of positive integers $\alpha = (\alpha_{1},\ldots,\alpha_{p})$ such that  $\alpha_{1} \geqslant \alpha_{2} \geqslant \ldots \geqslant \alpha_{p} > 0$ and $\alpha_{1} + \ldots + \alpha_{p} = s$. Each entry $\alpha_{i}$ is referred to as {\it part}, while the number of parts is referred to as {\it length} of a partition. One writes $|\lambda| = s$ and $\ell(\alpha) = p$. Extending the definition of partitions to $s = 0$ gives the {\it empty partition} $\varnothing$ with $\ell(\varnothing) = 0$. Denote $\mathcal{P}$ the set of partitions (including the empty partition) and $\mathcal{P}_{s} \subset \mathcal{P}$ the set of partitions of $s \geqslant 0$. For any subset $X \subseteq \mathcal{P}$, for any $N\geqslant 1$ one writes $X(N)$ to denote the set of elements of $X$ with at most $N$ parts. For pairs of partition we write $\mathcal{P}_{r,s} = \mathcal{P}_{r}\times \mathcal{P}_{s}$ and $\mathcal{P}_{r,s}(N) = \mathcal{P}_{r}(N)\times \mathcal{P}_{s}(N)$.
    \vskip 2pt
    
    For further convenience, given a partition $\alpha \in \mathcal{P}$ one defines $\alpha_{i}$ for any $i \in \mathbb{N}$ by setting $\alpha_{i} = 0$ for all $i > \ell(\alpha)$. With this convention at hand, for any $\alpha,\beta\in \mathcal{P}$ one defines $\alpha + \beta$ to be the partition with parts $\alpha_{i} + \beta_{i}$ for all $i\in\mathbb{N}$.
    \vskip 2pt

    Partitions admit a convenient graphical representation by Young diagrams. A non-empty partition $(\alpha_{1},\ldots,\alpha_{p})$ is represented by a left-justified array of $p$ rows of boxes, with $\alpha_{i}$ boxes in the $i$th row (we use the convention where $i$ increases downwards). With a slight abuse of notation, in the sequel we make no difference between partitions and Young diagrams. For example the set of partitions of $5$ with at most $3$ parts can be equally written as
    \begin{equation}
        \mathcal{P}_{5}(3) = \big\{(2_{2},1),(3,1_{2}),(3,2),(4,1),(5)\big\} \quad \text{or} \quad \mathcal{P}_{5}(3) = \big\{\ytableausetup{boxsize=7pt}\ydiagram{2,2,1},\, \ydiagram{3,1,1},\, \ydiagram{3,2},\, \ydiagram{4,1},\, \ydiagram{5}\big\}.
    \end{equation}
    Here and in what follows we make use of the shorthand notation where instead of $k$ equal parts $l > 0$ one writes $l_{k}$, so that in the above example one has $(2_{2},1) = (2,2,1)$ and $(3,1_{2}) = (3,1,1)$. 
    \vskip 2pt
    
    It is convenient to view Young diagrams as subsets of $\mathbb{N}^{2}$ and to write $(i,j) \in \alpha$ iff $1 \leqslant i \leqslant \ell(\alpha)$ and $1 \leqslant j \leqslant \alpha_{i}$. For any partition $\alpha\in\mathcal{P}$ one defines the dual partition $\alpha^{\prime}\in \mathcal{P}$ whose Young diagram is the set of all $(i,j)$ such that $(j,i)\in \alpha$. In words, the Young diagrams of $\alpha$ and $\alpha^{\prime}$ are related by transposition with respect to the main diagonal. 
    \vskip 2pt 
    
    Representing partitions as subsets of $\mathbb{N}^{2}$ also allows one to extend the following set-theoretic notions to partitions: given two partitions $\alpha,\beta\in \mathcal{P}$, intersection $\alpha\cap \beta\in \mathcal{P}$ and inclusion $\beta \subseteq \alpha$ are understood in terms of the corresponding subsets of $\mathbb{N}^{2}$.
    \vskip 2pt

    For two partitions $\alpha,\beta\in\mathcal{P}$ such that $\beta\subseteq \alpha$ one defines the {\it skew-shape} $\alpha\slash\beta$ represented by the Young-diagram $\alpha$ with the elements in $\beta$ crossed out, for example
    \begin{equation}
    \ytableausetup{boxsize=7pt}
	\alpha = \ydiagram{4,3,1,1}\;\;\text{and}\;\; \beta = \ydiagram{2,1,1}\quad\text{give}\quad\alpha\slash\beta = \ytableaushort{\times\times\xnone\xnone,\times\xnone\xnone,\times,\xnone}
    \end{equation}

    Given a non-empty partition $\alpha\in \mathcal{P}$, for any box $(i,j)\in \mathcal{P}$ one defines its content $c(i,j) = j - i$. The content of a partition is then
    \begin{equation}
        c(\alpha) = \sum_{(i,j)\in \alpha} c(i,j).
    \end{equation}
    In a natural way, the content of a skew-shape $\alpha\slash\beta$ is defined as $c(\alpha\slash\beta) = c(\alpha) - c(\beta)$.
    \vskip 2pt

    \paragraph{Characters of rational irreducible representations of $GL(N)$.} Given an integer $N \geqslant 1$ denote $\Lambda(N)\subset\mathcal{P}(N)^2$ the set of pairs of partitions $(\mu,\nu)$ such that $\ell(\mu) + \ell(\nu) \leqslant N$. The set $\Lambda(N)$ indexes all rational irreducible representations of $GL(N)$. In particular, the latter set indexes irreducible characters $s_{(\mu,\nu)}(x_1,\ldots,x_{N})$ of the rational representations of $GL(N)$ given by {\it rational Schur functions} (see \cite{Stembridge_87} and \cite[Proposition 2.7]{Koike_89}). Note that polynomial representations are rational representations with $\nu = \varnothing$, in which case the characters are Schur polynomials:
    \begin{equation}
        s_{(\mu,\varnothing)}(x_1,\ldots,x_{N}) = s_{\mu}(x_1,\ldots,x_{N}).
    \end{equation}
    The dual of the polynomial representation $(\mu,\varnothing)$ is the rational representation $(\varnothing,\mu)$ whose irreducible character is given by the following rational function:
    \begin{equation}
        s_{(\varnothing,\mu)}(x_1,\ldots,x_{N}) = s_{\mu}(x_1^{-1},\ldots,x_{N}^{-1}).
    \end{equation}
    More generally, rational representations $(\mu,\nu)$ and $(\nu,\mu)$ are dual.
    \vskip 2pt
    
    Any irreducible rational character is an irreducible polynomial character times a non-positive integer power of the determinant: for any $(\mu,\nu)\in \Lambda(N)$ there exists $\alpha\in \mathcal{P}(N)$ and $t\in \mathbb{N}_{0}$ such that
    \begin{equation}\label{eq:rational_polynomial_Schur}
        s_{(\mu,\nu)}(x_{1},\ldots,x_{N}) = \dfrac{1}{(x_{1}\ldots x_{N})^{t}} \, s_{\alpha}(x_{1},\ldots,x_{N}).
    \end{equation}    
    The choice of $\alpha$ and $t$ is not unique: the character rests intact upon passing to $t + s$ and $\alpha + (s_{N})$ for any $s\in \mathbb{N}_{0}$ (the latter transformation implies multiplying and dividing the character by $\det^{s}$, see also \cite[pp. 81--82]{Stembridge_87}). This suggests an alternative way of indexing rational irreducible representations of $GL(N)$ in terms of equivalence classes in $\mathcal{P}(N)\times\mathbb{N}_{0}$ defined as follows: two elements $(\alpha,t), (\beta,u)\in \mathcal{P}(N)\times\mathbb{N}_{0}$ are equivalent if and only if
    \begin{equation}
        \text{for all $i = 1,\ldots,N$,}\quad t - u = \alpha_{i} - \beta_{i}.
    \end{equation}
    Denote $\mathcal{S}(N)$ the so-defined set of equivalence classes and write $[\alpha,t]$ for the equivalence class with the representative $(\alpha,t)$. For example, for $N = 3$ and any integer $t\geqslant 0$ one has
    \begin{equation}
        \ytableausetup{boxsize=6pt} [\ydiagram{2},t] = \big\{(\ydiagram{2},t), (\ydiagram{3,1,1},t+1), (\ydiagram{4,2,2},t+2), \ldots \big\}.
    \end{equation}
    Given an equivalence class, we refer to the representative $(\alpha_{\min},t_{\min})$ with the minimal value of $t = t_{\min}$ as {\it minimal representative}. One has $\ell(\alpha_{\min}) < N$ whenever $t_{\min} > 0$. To see this note that if $t > t_{\min}$ then passing from $t$ to $t-1$ within an equivalence class amounts to omitting the leftmost column of height $N$ in $\alpha$.
    \vskip 2pt
    
    In order to link $\mathcal{S}(N)$ to $\Lambda(N)$, for any non-empty $\beta\in \mathcal{P}(N)$ with $\ell(\beta^{\prime}) = q$ define
    \begin{equation}\label{eq:map_bar}
        \bar{\beta} = (N - \beta^{\prime}_{q},\ldots,N - \beta^{\prime}_{1})^{\prime},\quad \text{otherwise set}\quad \bar{\varnothing} = \varnothing.
    \end{equation}
    In words, the Young diagram $\bar{\beta}$ is the set-theoretic difference $(q_{N})\backslash \beta$ rotated by $180^{\circ}$. The above map is utilised only in the context where $N$ is fixed, so the absence of $N$ in the notation does not lead to confusion. Finally, consider the following map:
    \begin{equation}\label{eq:map_staircase}
    \def\arraystretch{1.4}
        \begin{array}{rccc}
            \mathfrak{s} : & \Lambda(N) & \rightarrow & \mathcal{S}(N) \\
            \hfill & (\mu,\nu) & \mapsto & [\mu + \bar{\nu},\nu_{1}]
        \end{array}
    \end{equation}
    Note that $(\mu + \bar{\nu},\nu_{1})$ is the minimal representative. Indeed, if $\nu \neq \varnothing$ then $\ell(\mu + \bar{\nu}) \leqslant N-1$, while otherwise one has $\nu_{1} = 0$.
    \begin{lemma}\label{lem:bijection_staircase}
        The map \eqref{eq:map_staircase} is a bijection. 
    \end{lemma}
    \begin{proof}
        Let us describe the inverse map $\mathfrak{s}^{-1}$ by constructing $(\mu,\nu)\in \Lambda(N)$ from any $[\alpha,t] \in \mathcal{S}(N)$. In the sequel let $(\alpha,t)$ be the minimal representative.
        \vskip 2pt
        
        For $t = 0$ set $\mu = \alpha$ and $\nu = \varnothing$. Otherwise, let $t > 1$ so that $\alpha_N = 0$. If $t \geqslant \alpha_{1}$ then set $\mu = \varnothing$ and $\nu = (t - \alpha_{N},t - \alpha_{N-1},\ldots,t - \alpha_{1})$. In the opposite case one has $1 \leqslant t < \alpha_{1}$. Let $p$ be the maximal integer among $\{1,\ldots, N - 1\}$ such that $\alpha_{p} > t$, then set $\mu = (\alpha_{1} - t, \ldots,\alpha_{p} - t)$. Next, for the minimal integer $q$ among $\{p+1,\ldots, N\}$ such that $\alpha_{1} < t$ (which exists since $\alpha_{N} = 0$) set $\nu = (t - \alpha_{N},\ldots,t - \alpha_{q})$. As a matter of a direct check $\mathfrak{s}\big(\mathfrak{s}^{-1}[\alpha,t]\big) = [\alpha,t]$ and $\mathfrak{s}^{-1}\big(\mathfrak{s}(\mu,\nu)\big) = (\mu,\nu)$.
    \end{proof}

    The map \eqref{eq:map_staircase} admits a simple geometric interpretation which can be found for example in \cite{King_GenYTab}. The following example suffices to grasp the idea: for $N = 6$ take $\mu = (4,2,1)$ and $\nu = (3,2)$ so that $\mathfrak{s}(\mu,\nu)$ has the minimal representative with $\alpha = (7,5,4,3,1)$ and $t = 3$:

    \begin{equation}
    \ytableausetup{boxsize=8pt} \big(\, \ydiagram{4,2,1},\;\ydiagram{3,2}\,\big) \quad \xrightarrow{\phantom{m}\mathfrak{s}\phantom{m}}\quad \big[\begin{tikzpicture}[baseline=(current bounding box.center)]
            \def \l {0.3};
            \filldraw[lightgray] (0,0) -- (0,\l) -- (\l,\l) -- (\l,2*\l) -- (3*\l,2*\l) -- (3*\l,0) -- cycle ;
            \filldraw[lightgray] (3*\l,6*\l) -- (7*\l,6*\l) -- (7*\l,6*\l) -- (7*\l,5*\l) -- (5*\l,5*\l) -- (5*\l,4*\l) -- (4*\l,4*\l) -- (4*\l,3*\l) -- (3*\l,3*\l) -- cycle ;
            \draw (0,0) rectangle (3*\l,6*\l);
            \draw (0,\l) -- (3*\l,\l);
            \draw (\l,0) -- (\l,\l);
            \draw (2*\l,0) -- (2*\l,2*\l);
            \draw[double] (0,\l) -- (0,6*\l) -- (7*\l,6*\l) -- (7*\l,5*\l) -- (5*\l,5*\l) -- (5*\l,4*\l) -- (4*\l,4*\l) -- (4*\l,3*\l) -- (3*\l,3*\l) -- (3*\l,2*\l) -- (\l,2*\l) -- (\l,\l) -- cycle;
            \foreach[count = \i] \j in {1,3,4,5,7}
            {\draw (0,\l*\i) rectangle (\l*\j,{\l*(\i+1)});
            \foreach \k in {1,...,\j}
            {\draw (\k*\l,\l*\i) -- (\k*\l,{\l*(\i+1)});}
            }
            \draw[line width=0.7pt,<->] (0,-0.6*\l) -- (3*\l,-0.6*\l);
            \node at (-2.1*\l,3*\l) {$\scriptstyle{N = 6}$};
            \draw[line width=0.7pt,<->] (-0.6*\l,6*\l) -- (-0.6*\l,0);
            \node at (1.5*\l,-1.5*\l) {$\scriptstyle{t = 3}$};
    \end{tikzpicture}\;,\;3\big]
    \end{equation}
    \noindent Other way around, $\mu$ is obtained from $\alpha$ by omitting the $t$ leftmost columns, while $\nu$ is the set-theoretic difference $(t_N)\backslash \alpha$ rotated by $180^{\circ}$ (if $\alpha_{1} \leqslant t$ then $\mu = \varnothing$, if $t = 0$ then $\nu = \varnothing$).
    
    \paragraph{Character ring.} Integer combinations of polynomial characters of $GL(N)$ form an associative commutative unital ring: for any $\alpha,\beta\in \mathcal{P}(N)$ one has
    \begin{equation}\label{eq:LR_product}
        s_{\alpha}(x_{1},\ldots,x_{N}) \cdot s_{\beta}(x_{1},\ldots,x_{N}) = \sum_{\gamma\in \mathcal{P}(N)} c^{\gamma}_{\alpha\beta} s_{\gamma}(x_{1},\ldots,x_{N}),
    \end{equation}
    where the structure constants $c^{\gamma}_{\alpha\beta} \in\mathbb{N}_{0}$ are the {\it Littlewood-Richardson coefficients}. The latter admit a number of combinatorial descriptions in terms of operations with standard Young tableaux (see \cite{Fulton_YT}).
    \vskip 2pt

    From the relation between polynomial and rational characters \eqref{eq:rational_polynomial_Schur} via the map \eqref{eq:map_staircase} one reads off the structure of the associative commutative unital ring of integer combinations of rational Schur functions. The product rule follows from the extended RSK algorithm \cite{Stembridge_87} or, alternatively, from the products of universal characters \cite{Koike_89}. In the sequel we make use of the following particular product: for any non-empty $\rho,\sigma \in \mathcal{P}(N)$ consider
    \begin{equation}\label{eq:LR_ext_particular}
        s_{(\rho,\varnothing)}(x_{1},\ldots,x_{N})\cdot s_{(\varnothing,\sigma)}(x_{1},\ldots,x_{N}) = \sum_{(\mu,\nu)\in \Lambda(N)} c_{\rho\sigma}^{\mu\nu}(N)\,s_{(\mu,\nu)}(x_{1},\ldots,x_{N}),
    \end{equation}
    where the structure constants $c_{\rho\sigma}^{\mu\nu}(N) \in \mathbb{N}_0$ can be expressed in terms of the Littlewood-Richardson coefficients as follows. By recalling the relation of rational characters to Schur polynomials \eqref{eq:rational_polynomial_Schur} one has
    \begin{equation*}
            s_{(\rho,\varnothing)}(x_1,\ldots, x_N)\cdot s_{(\varnothing,\sigma)}(x_1,\ldots, x_N) = (x_1 \ldots x_N)^{-\sigma_{1}} s_{\rho}(x_1,\ldots, x_N)\cdot s_{\bar{\sigma}}(x_1,\ldots, x_N).
        \end{equation*}
    Then, by applying \eqref{eq:LR_product} one rewrites the product on the right-hand side of the above expression as follows:
    \begin{equation}\label{eq:LR_ext_poly}
        s_{(\rho,\varnothing)}(x_{1},\ldots,x_{N}) \cdot s_{(\varnothing,\sigma)}(x_{1},\ldots,x_{N}) = \sum_{\lambda\in \mathcal{P}(N)} c^{\lambda}_{\rho\bar{\sigma}}\,\dfrac{1}{(x_{1}\ldots x_{N})^{\sigma_{1}}}\,s_{\lambda}(x_{1},\ldots,x_{N}).
    \end{equation}
    By \eqref{eq:map_staircase} one identifies $s_{(\mu,\nu)}(x_{1},\ldots,x_{N}) =(x_{1}\ldots x_{N})^{-\sigma_{1}}\,s_{\lambda}(x_{1},\ldots,x_{N})$ such that $\mathfrak{s}^{-1}[\lambda,\sigma_{1}] = (\mu,\nu)$, so by comparing \eqref{eq:LR_ext_particular} and \eqref{eq:LR_ext_poly} one has:
    \begin{equation}\label{eq:multiplicity_LR_coef}
    \def\arraystretch{1.3}
        c_{\rho\sigma}^{\mu\nu}(N) = 
        \left\{
        \begin{array}{rl}
            c^{\lambda}_{\rho\bar{\sigma}}, & \text{such that $(\mu,\nu) = \mathfrak{s}^{-1}[\lambda,\sigma_{1}]$ if $\lambda$ occurs on the right-hand side of \eqref{eq:LR_ext_poly}}\\
            0, & \text{otherwise}
        \end{array}
        \right.
    \end{equation}
    Another description of the above structure constants \eqref{eq:rational_polynomial_Schur} is given in \cite[Corollary 2.3.1]{Koike_89} in terms of universal characters.

\section{Traceless projection of mixed tensors}\label{sec:one}

    \subsection{Traceless projector}\label{sec:traceless_projection}
    
    \paragraph{Mixed tensor products.} Let $V$ be a finite-dimensional $\mathbb{C}$-vector space of dimension $N \geqslant 1$, and let $V^{*}$ denote the dual space. For any fixed integers $m,n \geqslant 0$ consider the space of $m$-contravariant and $n$-covariant tensors
    \begin{equation}\label{eq:mixed_tensors}
        V^{m,n} = V^{\otimes m}\otimes V^{*\otimes n},
    \end{equation}
    where by definition $V^{0,0} = \mathbb{C}$. In the case when both $m,n\geqslant 1$ one says that \eqref{eq:mixed_tensors} is a {\it mixed tensor product}. By fixing a basis $\{e_{i}\}$ in $V$, as well as the canonical dual basis $\{e^{i}\}$ in $V^{*}$, each tensor $T\in V^{m,n}$ is identified with the set of its components:
    \begin{equation}\label{eq:mixed_basis}
        T = t^{i_1\ldots i_{m}}_{\phantom{i_1\ldots i_{m}} j_{1} \ldots j_{n}} \, e_{i_1}\otimes \ldots \otimes e_{i_m}\otimes e^{j_1} \otimes \ldots \otimes e^{j_n}\,.
    \end{equation}
    Here and in the sequel for each pair of repeated upper and lower indices summation is implied. In particular, $V^{1,1} = V \otimes V^{*}$ is isomorphic to the space of endomorphisms of $V$, with the trace subspace spanned by the identity operator:
    \begin{equation}\label{eq:cotrace}
        E = \sum_{i = 1}^{N} e_{i} \otimes e^{i} = \delta^{i}_{\phantom{i}j} \,e_{i}\otimes e^{j}.
    \end{equation}

    \paragraph{Traceless tensors.} Consider the following contraction maps: for any $1\leqslant \sfa\leqslant m$ and $1^{\prime} \leqslant \hsfb\leqslant n^{\prime}$ define
    \begin{equation}\label{eq:map_trace}
        \mathrm{tr}_{\sfa\hsfb} : V^{m,n} \to V^{m-1,n-1}
    \end{equation}
    such that 
    \begin{equation}\label{eq:map_trace_components}
        \mathrm{tr}_{\sfa\hsfb} (v_{1}\otimes \ldots\otimes v_{m}\otimes \varphi_{1} \otimes\ldots\otimes \varphi_{n}) = \langle \varphi_{\hsfb},v_{\sfa}\rangle \, (v_{1}\otimes \ldots \otimes \mathrlap{\bcancel{\phantom{\scalebox{1.3}{v}}}}{v_{\sfa}} \otimes\ldots \otimes v_{m}\otimes \varphi_{1} \otimes \ldots \otimes \mathrlap{\bcancel{\phantom{\scalebox{1.3}{n}}}}{\varphi_{\hsfb}} \otimes\ldots \otimes \varphi_{n}).
    \end{equation}
    In words, one contracts the $\sfa$th vector with the $\hsfb$th covector. Equivalently, in terms of tensor components, one contracts the $\sfa$th upper index with the $\hsfb$th lower index:
    \begin{equation}
        \mathrm{tr}_{\sfa\sfb} \,:\, t^{i_{1}\ldots i_{\sfa}\ldots i_{m}}_{\phantom{i_{1}\ldots i_{\sfa}\ldots i_{m}} j_{1}\ldots j_{\hsfb} \ldots j_{n}} \,\mapsto\, t^{i_{1}\ldots k\ldots i_{m}}_{\phantom{i_{1}\ldots k\ldots i_{m}} j_{1}\ldots k \ldots j_{n}}
    \end{equation}
    \vskip 4pt
    
    We complement \eqref{eq:map_trace} by the set of insertion maps: for all 
    $1 \leqslant \sfa \leqslant m$ and $1 \leqslant \hsfb \leqslant n$ define
    \begin{equation}\label{eq:map_cotrace}
        \mathrm{tr}^{+}_{\sfa\hsfb} : V^{m-1,n-1} \to V^{m,n}
    \end{equation}
    such that 
    \begin{equation}\label{eq:map_cotrace_components}
        \mathrm{tr}^{+}_{\sfa\hsfb} (v_{1}\otimes \ldots\otimes v_{m-1}\otimes \varphi_{1} \otimes\ldots\otimes \varphi_{n-1}) = \sum_{i=1}^{N} (v_{1}\otimes \ldots \otimes v_{\sfa-1} \otimes e_{i} \otimes\ldots \otimes v_{m-1}\otimes \varphi_{1} \otimes \ldots \otimes \varphi_{\hsfb-1} \otimes e^{i} \otimes\ldots \otimes \varphi_{n-1}).
    \end{equation}
    In words, one inserts \eqref{eq:cotrace} such that the corresponding basis vectors $\{e_{i}\}$ and $\{e^{i}\}$ occur at the $\sfa$th and $\hsfb$th positions respectively. In terms of tensor components, one multiplies each component by the Kronecker delta, with its contravariant and covariant indices occurring at the $\sfa$th and $\hsfb$th positions respectively:
    \begin{equation}
        \mathrm{tr}^{+}_{\sfa\hsfb} \,:\, t^{i_1\ldots i_{\sfa-1} i_{\sfa}\ldots i_{m-1}}{}_{j_{1}\ldots j_{\hsfb-1} j_{\hsfb}\ldots j_{n-1}} \,\mapsto \, \delta^{i_{\sfa}}_{
        \phantom{i} j_{\hsfb}} t^{i_1\ldots i_{\sfa-1}i_{\sfa+1}\ldots i_{m}}{}_{j_{1}\ldots j_{\hsfb-1} j_{\hsfb+1}\dots j_{n}}
    \end{equation}
    As a matter of a routine exercise, one checks that the map \eqref{eq:map_cotrace} is independent of the choice of a basis in $V$.
    \vskip 2pt
    
    One has the following family of endomorphisms of $V^{m,n}$ for all $1\leqslant \sfa \leqslant m$, $1^{\prime}\leqslant \hsfb \leqslant n^{\prime}$:
    \begin{equation}\label{eq:arcs_image}
        \tau_{\sfa\hsfb} = \ctr{\sfa}{\hsfb}\circ \tr{\sfa}{\hsfb},
    \end{equation}
    whose action on simple tensors reads as follows:
    \begin{equation}
    \def\arraystretch{1.4}
    \begin{array}{l}
        \tau_{\sfa\hsfb} (v_{1}\otimes \ldots \otimes v_{\sfa}\otimes \ldots \otimes v_{m}\otimes \varphi_{1} \otimes \ldots \otimes \varphi_{\sfb^{\prime}} \otimes \ldots\otimes \varphi_{n}) \\
        \displaystyle \phantom{mmmmmmmmm} = \langle \varphi_{\sfb^{\prime}}, v_{\sfa}\rangle\,\sum_{k = 1}^{N} v_{1}\otimes \ldots \otimes e_{k}\otimes \ldots \otimes v_{m}\otimes \varphi_{1} \otimes \ldots \otimes e^{k} \otimes \ldots\otimes \varphi_{n}.
    \end{array}
    \end{equation}

    \paragraph{Traceless subspace.} For any fixed $m,n\geqslant 1$ define the {\it traceless subspace} $V^{m,n}_{0} \subset V^{m,n}$ to be the common kernel of the trace maps:
    \begin{equation}
        V^{m,n}_{0} = \bigcap_{1 \leqslant \sfa \leqslant m, \, 1 \leqslant \hsfb \leqslant n} \mathrm{ker}(\mathrm{tr}_{\sfa\hsfb}),
    \end{equation}
    and denote
    \begin{equation}
        V_{1}^{m,n} = \bigg\langle \bigcup_{1\leqslant \sfa \leqslant m,\,1\leqslant \hsfb \leqslant n}\mathrm{Im}( \ctr{\sfa}{\hsfb})  \bigg\rangle,
    \end{equation}
    the space spanned by the image of \eqref{eq:map_cotrace}. Also consider the following endomorphism of $V^{m,n}$:
    \begin{equation}\label{eq:operator_A_image}
        \scrA_{m,n} = \sum_{1\leqslant\sfa\leqslant m,\;1^{\prime}\leqslant\hsfb\leqslant n^{\prime}} \tau_{\sfa\hsfb} ,
    \end{equation}
    which plays a key role in the sequel, and denote $\mathrm{spec}(\scrA_{m,n})$ the set of eigenvalues of $\scrA_{m,n}$. The proof of the following lemma is delegated to Appendix \bref{sec:proof_lemma_A}.

    \begin{lemma}\label{lem:operator_A}
    For any $m,n\geqslant 1$ and $N \geqslant 1$ the endomorphism $\scrA_{m,n}$ of $V^{m,n}$ has the following properties:
    \begin{itemize}
        \item[1)] $\scrA_{m,n}$ is diagonalisable,
        \item[2)] $\mathrm{ker} (\scrA_{m,n}) = V^{m,n}_{0}$ and $\mathrm{Im} (\scrA_{m,n}) = V_{1}^{m,n}$ so that $V^{m,n} = V_{0}^{m,n} \oplus V_{1}^{m,n}$,
        \item[3)] $\mathrm{spec}(\scrA_{m,n})\subseteq \mathbb{R}_{\geqslant 0}$.
    \end{itemize}
    \end{lemma}

    As a straightforward corollary of the above lemma for $N = 1$, $V_{0}^{1,1} = \{0\}$ because $V^{1,1} \subseteq V^{1,1}_{1}$. In this case the traceless projection is zero, and thus \eqref{eq:traceless_projector} expresses the zero operator. Nevertheless for the sake of completeness of the exposition we keep the possibility $N = 1$ in the sequel.
    
    \paragraph{Traceless projector.} Projection of $V^{m,n}$ onto $V^{m,n}_{0}$ is referred to as {\it traceless projection}. With the point (2) of Lemma \bref{lem:operator_A} at hand we focus on the following particular traceless projector:
    \begin{equation}\label{eq:traceless_projector}
        \mathscr{P}_{m,n} : V^{m,n} \to V^{m,n}_{0}\quad\text{such that}\quad \mathrm{ker}(\mathscr{P}_{m,n}) = V_{1}^{m,n}.
    \end{equation}
    Since the above projector maps $V^{m,n}$ onto $\mathrm{ker}(\scrA_{m,n})$ and annihilates $\mathrm{Im}(\scrA_{m,n})$, it admits the simple factorised form \eqref{eq:intro_traceless_projector_factorised}. The set $\mathrm{spec}(\scrA_{m,n})$ can be constructed explicitly, by means of simple manipulations with Young diagrams via the following algorithm (recall the relevant notations and operations with Young diagrams given in Section \bref{sec:intro_notations}).
    \vskip 2pt

    \begin{itemize}
        \item[{\it Step 1.}] Write down all the pairs $(\rho,\sigma)\in \mathcal{P}_{m,n}(N)$.
        \item[{\it Step 2.}] For each pair construct $\bar{\sigma}$ (which may be empty) and apply the Littlewood-Richardson rule to determine all the Young diagrams $\lambda$ with at most $N$ rows such that $c^{\lambda}_{\rho \bar{\sigma}} \neq 0$.
        \item[{\it Step 3.}] For each $\lambda$ construct $(\mu,\nu) = \mathfrak{s}^{-1}[\lambda,\sigma_{1}]$, as well as the skew-shape diagrams $\rho\slash\mu$ and $\sigma\slash\nu$.
        \item[{\it Step 4.}] For each pair of skew-shapes calculate $r = |\rho| - |\mu| = |\sigma| - |\nu|$, so that the corresponding eigenvalue of $\scrA_{m,n}$ equals $Nr + c(\rho\slash\mu) + c(\sigma\slash\nu)$.
    \end{itemize}
    
    \noindent The eigenvalues obtained throughout the above steps exhaust $\mathrm{spec}(\scrA_{m,n})$. Explanation and proof of this result constitutes the subject of the forthcoming Section \bref{sec:traceless_rep_theory}.

    \paragraph{Examples of traceless projectors.} Let us apply the above algorithm to a number of simple cases. In the simplest case of $m = n = 1$ and $N \geqslant 2$, by {\it Step 1} one identifies the only pair $\ytableausetup{boxsize=5.5pt}\rho = \ydiagram{1}$ and $\ytableausetup{boxsize=5.5pt}\sigma = \ydiagram{1}$. By {\it Step 2} one has $\bar{\sigma} = (1_{N-1})$, in which case $c^{\lambda}_{\rho\bar{\sigma}} \neq 0$ for $\lambda = (2,1_{N-2})$ and $(1_{N})$. Calculating $\mathfrak{s}^{-1}[\lambda,1]$ according to {\it Step 3} gives $\ytableausetup{boxsize=5.5pt}\mu = \ydiagram{1}$ and $\ytableausetup{boxsize=5.5pt}\nu = \ydiagram{1}$ for the former case and $\mu = \varnothing$ and $\nu = \varnothing$ for the latter. By {\it Step 4} one has the following two eigenvalues:
    \begin{equation}
        \ytableausetup{boxsize=7pt} N\cdot 0 + c(\ytableaushort{\times}) + c(\ytableaushort{\times}) = 0\;\;\text{(for $r = 0$)}\quad \text{and}\quad N\cdot 1 + c(\ytableaushort{\xnone}) + c(\ytableaushort{\xnone}) = N\;\;\text{(for $r = 1$)}
    \end{equation}
    and thus
    \begin{equation}
        \mathrm{spec}(\scrA_{1,1}) = \{0,N\}.
    \end{equation}
    Since $\scrA_{1,1} = \tau_{1,1^{\prime}}$ one readily recognises the well-known traceless projection of a square $N$-by-$N$ matrix: 
    \begin{equation}\label{eq:traceless_projector_11}
        \mathscr{P}_{1,1} = 1 - \dfrac{1}{N}\tau_{1,1^{\prime}},\quad \text{such that}\quad \scrP_{1,1} \,:\, t^{i}{}_{j} \mapsto t^{i}{}_{j} - \dfrac{1}{N} \delta^{i}_{j} \;t^{k}{}_{k}.
    \end{equation}
    For $N = 1$ one has $\bar{\sigma} = \varnothing$ and $\mu = \nu = \varnothing$. As a result, $\mathrm{spec}(\scrA_{1,1}) = \{1\}$, so again \eqref{eq:traceless_projector_11} is the sought traceless projector. As was already mentioned below Lemma \bref{lem:operator_A}, in the case $N = 1$ the traceless subspace is trivial. This equally follows by observing that $\tau_{1,1^{\prime}}(e_{1}\otimes e^{1}) = e_{1}\otimes e^{1}$, so that \eqref{eq:traceless_projector_11} indeed annihilates $V^{1,1}$.
    \vskip 2pt

    Another simple example is for $m = 2$, $n = 1$. In order to avoid any restrictions on possible Young diagrams at first instance, suppose $N \geqslant 3$. For $\ytableausetup{boxsize=5.5pt}\rho = \ydiagram{1,1}$ and $\ytableausetup{boxsize=5.5pt}\sigma = \ydiagram{1}$ one has $\bar{\sigma} = (1_{N-1})$. By applying the Littlewood-Richardson rule, $c^{\lambda}_{\rho\bar{\sigma}} \neq 0$ for $\lambda = (2,2,1_{N-3})$ and $(2,1_{N-1})$. In the former case one has $\ytableausetup{boxsize=5.5pt} \mu = \ydiagram{1,1}$, $\ytableausetup{boxsize=5.5pt}\nu = \ydiagram{1}$ (with $r = 0$), while the latter gives $\ytableausetup{boxsize=5.5pt}\mu = \ydiagram{1}$, $\nu = \varnothing$ (with $r = 1$), so that the corresponding eigenvalues are
    \begin{equation}
        \ytableausetup{boxsize=7pt} N\cdot 0 + c(\ytableaushort{\times,\times}) + c(\ytableaushort{\times}) = 0\;\;\text{(for $r = 0$)}\quad \text{and}\quad N\cdot 1 + c(\ytableaushort{\times,\xnone}) + c(\ytableaushort{\xnone}) = N - 1\;\;\text{(for $r = 1$)}.
    \end{equation}
    By repeating the same steps for $\ytableausetup{boxsize=5.5pt} \rho = \ydiagram{2}$ and $\ytableausetup{boxsize=5.5pt} \sigma = \ydiagram{1}$ one recovers the two diagrams $(3,1_{N-2})$ and $(2,1_{N-1})$, so that $\ytableausetup{boxsize=5.5pt} \mu = \ydiagram{2}$, $\ytableausetup{boxsize=5.5pt} \nu = \ydiagram{1}$ (with $r = 0$) for the former and $\ytableausetup{boxsize=5.5pt} \mu = \ydiagram{1}$, $\nu = \varnothing$ (with $r = 1$) for the latter. The corresponding eigenvalues are
    \begin{equation}
        \ytableausetup{boxsize=7pt} N\cdot 0 + c(\ytableaushort{\times\times}) + c(\ytableaushort{\times}) = 0\;\;\text{(for $r = 0$)}\quad \text{and}\quad N\cdot 1 + c(\ytableaushort{\times\xnone}) + c(\ytableaushort{\xnone}) = N + 1\;\;\text{(for $r = 1$)}.
    \end{equation}
    All in all one has
    \begin{equation}\label{eq:spec_A_21}
        \mathrm{spec}(\scrA_{2,1}) = \{0,N-1\} \cup\{0,N+1\} = \{0,N-1,N+1\}.
    \end{equation}
    Omitting details, let us mention that the same spectrum takes place for $N = 2$, thus
    \begin{equation}\label{eq:traceless_projector_21}
        \mathscr{P}_{2,1} = \left(1 - \dfrac{1}{N-1}\,\scrA_{2,1}\right)\left(1 - \dfrac{1}{N+1}\,\scrA_{2,1}\right)\quad \text{for all}\; N\geqslant 2.
    \end{equation}
    Direct computation gives
    \begin{equation}
        \scrP_{2,1} \,:\, t^{ij}{}_{k} \mapsto t^{ij}{}_{k} - \dfrac{N}{(N-1)(N+1)}\left(\delta^{i}_{k}\;t^{pj}{}_{p} + \delta^{j}_{k}\;t^{ip}{}_{p}\right) + \dfrac{1}{(N-1)(N+1)} \left(\delta^{i}{}_{k} \; t^{jp}{}_{p} + \delta^{j}{}_{k} \; t^{pi}{}_{p}\right).
    \end{equation}
    \vskip 2pt
    
    The above expression simplifies upon assuming permutation symmetry in the contravariant indices. For a symmetric tensor with components $s^{ij}{}_{k} = s^{ji}{}_{k}$ the traceless projection reduces to
    \begin{equation}\label{eq:traceless_projector_21_s}
        \scrP_{2,1} \,:\, s^{ij}{}_{k} \mapsto s^{ij}{}_{k} - \dfrac{1}{N+1}\left(\delta^{i}_{k}\;s^{jp}{}_{p} + \delta^{j}_{k}\;s^{ip}{}_{p}\right),
    \end{equation}
    which results from application of the only factor with the eigenvalue $N+1$. Similarly, for an anti-symmetric tensor with components $a^{ij}{}_{k} = -a^{ji}{}_{k}$ the traceless projection reads as follows
    \begin{equation}\label{eq:traceless_projector_21_a}
        \scrP_{2,1} \,:\, a^{ij}{}_{k} \mapsto a^{ij}{}_{k} - \dfrac{1}{N-1}\left(\delta^{j}_{k}\;a^{ip}{}_{p} - \delta^{i}_{k}\;a^{jp}{}_{p}\right),
    \end{equation}
    and is due to applying the factor with the eigenvalue $N-1$. The possibility of reducing the number of factors in \eqref{eq:intro_traceless_projector_factorised} is explained in Section \bref{sec:projector_further} by means of restricted traceless projectors \eqref{eq:traceless_projector_restricted}.

    \subsection{Group action in mixed tensor products, and its centraliser algebra}\label{sec:traceless_rep_theory}

    \paragraph{Group action.} The automorphism group of $V$ is the full complex linear group $GL(N)$. The latter acts on the dual space $V^{*}$ contragrediently such that the canonical pairing is $GL(N)$-invariant: for all $v\in V$, $\varphi\in V^{*}$ and $S\in GL(N)$
    \begin{equation}\label{eq:pairing}
        \langle S(\varphi),S(v)\rangle = \langle \varphi,v \rangle .
    \end{equation}
    In this way, $V$ and $V^{*}$ are two inequivalent irreducible representations of $GL(N)$. Given a basis $\{e_{i}\}$ of $V$, as well as the canonical dual basis $\{e^{i}\}$ of $V^{*}$, any element $S\in GL(N)$ is represented by an invertible matrix which gives the decomposition of the image of the basis vectors:
    \begin{equation}\label{eq:action_basis}
        S(e_{i}) = e_{j}\,S^{j}{}_{i}\quad \text{and}\quad S(e^{i}) = (S^{-1})^{i}{}_{j}\,e^{j}.
    \end{equation}
    Note in this respect that the complement $V^{m,n}_{1}$ is $GL(N)$-invariant.
    \vskip 2pt
    
    By assuming the diagonal action of $GL(N)$ on tensors, for each $m,n \geqslant 1$ the mixed tensor product $V^{m,n}$ is a rational representation of $GL(N)$:
    \begin{equation}
        S(v_{1}\otimes \ldots \otimes v_{m}\otimes\varphi_{1}\otimes\ldots \otimes \varphi_{n}) = S(v_1)\otimes \ldots \otimes S(v_{m}) \otimes S(\varphi_{1}) \otimes \ldots \otimes S(\varphi_{n}),
    \end{equation}
    which means that matrix elements of a $GL(N)$-transformation are rational functions of group parameters. The classical result is that rational representations of $GL(N)$ are completely reducible \cite{Schur_dissertation}. 
    \vskip 2pt
    
    Irreducible rational representations occurring in $V^{m,n}$ are indexed by the set (see, {\it e.g.}, \cite{King_GenYTab,Stembridge_87,Koike_89})
    \begin{equation}
        \Lambda_{m,n}(N) = \bigcup_{r = 1}^{\min(m,n)}\Lambda_{m,n}^{(r)}(N)
    \end{equation}
    where for each $r\in \{0,1,\ldots,\min(m,n)\}$
    \begin{equation}\label{eq:index_modules_GL}
        \Lambda_{m,n}^{(r)}(N) = \big\{(\mu,\nu)\in \mathcal{P}^2\;:\; m - |\mu| = n - |\nu| = r\,,\;\; \ell(\mu) + \ell(\nu) \leqslant N\big\}.
    \end{equation}
    For example,
    \begin{equation}\label{eq:partitions}
    \ytableausetup{boxsize=6pt}
        \Lambda^{(0)}_{3,2}(3) = \big\{ (\ydiagram{3},\ydiagram{1,1}),\; (\ydiagram{2,1},\ydiagram{2}),\; (\ydiagram{3},\ydiagram{2})\big\},\quad \Lambda^{(1)}_{3,2}(3) = \big\{ (\ydiagram{1,1},\ydiagram{1}),\; (\ydiagram{2},\ydiagram{1})\big\},\quad \Lambda^{(2)}_{3,2}(3) = \big\{ (\ydiagram{1},\varnothing)\big\},
        \def\arraystretch{0.3}
    \end{equation}
    where one notes the absence of $(\ytableausetup{boxsize=6pt}\ydiagram{2,1},\ydiagram{1,1})$ and $(\ytableausetup{boxsize=6pt}\ydiagram{1,1,1},\ydiagram{1,1})$ in $\Lambda^{(0)}_{3,2}(3)$ due to $\ell(\mu) + \ell(\nu) > 3$.
    \vskip 2pt
    
    For each $(\mu,\nu)\in \Lambda_{m,n}(N)$ denote $U^{(\mu,\nu)}$ the corresponding irreducible rational representation of $GL(N)$, then one has the following decomposition of the mixed tensor product (see \cite[Corollary 4.7]{Stembridge_87} and \cite[Theorem 1.1]{Koike_89}).
    
    \begin{theorem}\label{thm:tensor_decomposition}
    For any $m,n\geqslant 1$ and $N\geqslant 1$ the mixed tensor product $V^{m,n}$ decomposes as a direct sum of irreducible rational representations of $GL(N)$:
    \begin{equation}\label{eq:decomposition_GL}
        V^{m,n} \cong \bigoplus_{(\mu,\nu) \in \Lambda_{m,n}(N)} (U^{(\mu,\nu)})^{\oplus d_{\mu\nu}},\quad d_{\mu\nu} \geqslant 1.
    \end{equation}
    The traceless subspace $V_{0}^{m,n} \subset V^{m,n}$ decomposes as a direct sum of irreducible rational representations of $GL(N)$ indexed by pairs of partitions of $m$ and $n$:
    \begin{equation}\label{eq:traceless_subspace_decomposition}
        V^{m,n}_{0} = \bigoplus_{(\mu,\nu) \in \Lambda^{(0)}_{m,n}(N)} (U^{(\mu,\nu)})^{\oplus d_{\mu\nu}}
    \end{equation}
    As a result, the $GL(N)$-invariant complement $V^{m,n}_{1} \subseteq V^{m,n}$ is uniquely-defined.
    \end{theorem}
    
    \paragraph{Centraliser algebra.} Projection onto the $GL(N)$-invariant subspace $V_{0}^{m,n}$ along the $GL(N)$-invariant complement $V_{1}^{m,n}$ commutes with the action of $GL(N)$, and hence the corresponding projector \eqref{eq:traceless_projector} belongs to the centraliser algebra $C_{m,n}(N) = \mathrm{End}_{GL(N)}(V^{(m,n)})$:
    \begin{equation}
        \mathscr{P}_{m,n} \in C_{m,n}(N).
    \end{equation}
    \vskip 4pt
    
    The algebra $C_{m,n}(N)$ contains the image of the symmetric group algebra $\mathbb{C}[\Sn{m}\times \Sn{n}]$ generated by transpositions $\tau_{\sfa\sfb}$ and $\tau_{\hsfa\hsfb}$ (for all $1 \leqslant \sfa < \sfb \leqslant m$ and $1^{\prime} \leqslant \hsfa < \hsfb \leqslant n^{\prime}$):
    \begin{equation}\label{eq:transpositions}
    \def\arraystretch{1.4}
    \begin{array}{l}
        \tau_{\sfa\sfb}(v_{1} \otimes \ldots \otimes v_{\sfa} \otimes \ldots \otimes v_{\sfb}\otimes\ldots \otimes v_{m} \otimes \varphi^{1^{\prime}}\otimes \ldots \otimes \varphi^{n^{\prime}})\\
        \phantom{mmmmmmmmmmmm} =\; v_{1} \otimes \ldots \otimes v_{\sfb} \otimes \ldots \otimes v_{\sfa}\otimes\ldots \otimes v_{m} \otimes \varphi^{1^{\prime}}\otimes \ldots \otimes \varphi^{n^{\prime}},\\
        \tau_{\hsfa\hsfb}(v_{1}\otimes \ldots \otimes v_{m}\otimes \varphi^{1} \otimes \ldots \otimes \varphi^{\hsfa} \otimes \ldots \otimes \varphi^{\hsfb}\otimes\ldots \otimes \varphi^{n})\\
        \phantom{mmmmmmmmmmmm} =\; v_{1}\otimes \ldots \otimes v_{m}\otimes \varphi^{1^{\prime}} \otimes \ldots \otimes \varphi^{\hsfb} \otimes \ldots \otimes \varphi^{\hsfa}\otimes\ldots \otimes \varphi^{n^{\prime}}.
    \end{array}
    \end{equation}
    The whole algebra $C_{m,n}(N)$ is finitely-generated due to the following result [Koike, Lemma 1.2].
    \begin{theorem}\label{thm:cenraliser_generate}
        For any $m,n\geqslant 1$ and $N \geqslant 1$, for all $1 \leqslant \sfa < \sfb \leqslant m$ and $1^{\prime} \leqslant \hsfa < \hsfb \leqslant n^{\prime}$ the endomorphisms $\tau_{\sfa\sfb}$, $\tau_{\hsfa\hsfb}$ and $\tau_{\sfa\hsfb}$ generate the centraliser algebra $C_{m,n}(N)$ of the diagonal action of $GL(N)$ in the space of mixed tensors $V^{m,n}$.
    \end{theorem}
    \noindent As an immediate consequence of the above theorem,
    \begin{equation}
        \scrA_{m,n} \in C_{m,n}(N)\,,
    \end{equation}
    which indeed implies that $\mathscr{P}_{m,n} \in C_{m,n}(N)$ manifestly in view of \eqref{eq:intro_traceless_projector_factorised}. Also note that the factorised form \eqref{eq:intro_traceless_projector_factorised} implies the following:
    \begin{equation}\label{eq:traceless_projector_X}
        \scrP_{m,n} = 1 + \mathscr{X}_{m,n}\,,
    \end{equation}
    where $\mathscr{X}_{m,n}$ is spanned by products of the trace generators $\tau_{\sfa\hsfb}$.
    \vskip 2pt
    
    Finally, by combining Theorem \bref{thm:tensor_decomposition} with Lemma \bref{lem:operator_A}, the traceless projector \eqref{eq:intro_traceless_projector_factorised} acts by identity in any isotypic $GL(N)$-component if $(\mu,\nu)\in \Lambda^{(0)}_{m,n}$, and annihilates it otherwise, which implies the following.
    \begin{proposition}\label{prop:P_central}
    For any $m,n\geqslant 1$ and $N\geqslant 1$,
    \begin{equation}
        \mathscr{P}_{m,n} \in Z\big(C_{m,n}(N)\big).
    \end{equation}
    In particular, $\mathscr{P}_{m,n}$ commutes with the action of $\mathbb{C}[\Sn{m}\times\Sn{n}]$ in $V^{m,n}$, generated by the transpositions \eqref{eq:transpositions}.
    \end{proposition}

    \paragraph{Casimir element.} For any $m,n \geqslant 0$ consider the following elements of $C_{m,n}(N)$:
    \begin{equation}\label{eq:JM}
    \def\arraystretch{1.4}
        \scrL_{m} = \left\{\begin{array}{rl}
            \displaystyle\sum_{1\leqslant \sfa < \sfb \leqslant m} \tau_{\sfa\sfb}, & m \geqslant 2 \\
            0, & \text{otherwise} 
        \end{array}
        \right.
        \quad\text{and}\quad \scrR_{n} = \left\{
        \begin{array}{rl}
            \displaystyle\sum_{1^{\prime}\leqslant \hsfa < \hsfb \leqslant n^{\prime}} \tau_{\hsfa\hsfb}, & n \geqslant 2 \\
            0, & \text{otherwise}
        \end{array}
        \right.
    \end{equation}
    and in the case when $m,n \geqslant 1$ consider
    \begin{equation}\label{eq:JMC}
        \scrC_{m,n} = \scrL_{m} + \scrR_{n} - \scrA_{m,n} + Nn,
    \end{equation}
    where the constant term is introduced for the sake of convenience. Behind the endomorphisms $\scrL_{m}$ and $\scrR_{n}$ one recognises the sum of the Jucys-Murphy elements in $\mathbb{C}\Sn{m}$ and $\mathbb{C}\Sn{n}$ respectively \cite{Jucys,Murphy}. The same observation holds for $\scrC_{m,n}$, where the ``background'' associative algebra is the walled Brauer algebra, and the analogues of the Jucyc-Murphy elements were introduced in \cite{JungKim_2020}. However in what follows we derive all necessary knowledge about \eqref{eq:JMC} directly from its action on tensors in combination with the classical facts about \eqref{eq:JM}.
    \vskip 2pt 
    
    Along the same lines as in the proof of \cite[Theorem 2.6]{Nazarov} one relates the homomorphism \eqref{eq:JMC} to a particular central element in $U(\mathfrak{gl}_N)$. Namely, consider the matrix units $\{E^{i}_{\phantom{i}j}\} \subset \mathfrak{gl}_N$ such that
    \begin{equation}
        E^{i}_{\phantom{i}j}(e_{k}) = -e_{j}\delta^{i}_{k} \quad \text{and}\quad E^{i}_{\phantom{i}j}(e^{k}) = \delta_{j}^{k} e^{i}.
    \end{equation}
    Recall that $V^{m,n}$ is a $U(\mathfrak{gl}_N)$-module where an element of $\mathfrak{gl}_N$ acts as a derivation, and consider the following Casimir element:
    \begin{equation}\label{eq:Casimir}
        \dfrac{1}{2} \big(E^{i}_{\phantom{i}j} E^{j}_{\phantom{i}i} + N E^{i}_{\phantom{i}i}\big) \in Z\big(U(\mathfrak{gl}_{N})\big).
    \end{equation}
    The following assertion is checked directly on the basis of $V^{m,n}$.
    \begin{lemma}\label{lem:Casimir}
    For any $m,n \geqslant 0$ and $N\geqslant 1$, for any $T \in V^{m,n}$ one has:
        \begin{equation}
        \def\arraystretch{1.4}
            \dfrac{1}{2} \big(E^{i}_{\phantom{i}j} E^{j}_{\phantom{i}i} + N E^{i}_{\phantom{i}i}\big)\, T = \left\{
                \begin{array}{rl}
                    \scrC_{m,n}(T)\,, & m,n\geqslant 1\\
                    \scrL_{m}(T)\,, & n = 0\\
                    \scrR_{n}(T) + Nn\,T\,, & m = 0
                \end{array}
            \right.
        \end{equation}
        In particular $\scrC_{m,n} \in Z\big(C_{m,n}(N)\big)$.
    \end{lemma}
    \noindent With the aid of Lemma \bref{lem:Casimir} one proves the following theorem (see \cite[Theorem 2.6]{Nazarov} and recall Theorem \bref{thm:tensor_decomposition}).
    \begin{proposition}\label{prop:eigenvalues_Casimir}
        For any $m,n \geqslant 1$ and $N\geqslant 1$, let $(\mu,\nu) \in \Lambda_{m,n}(N)$. Then for any tensor $T\in V^{m,n}$ in the isotypic $GL(N)$-component $(\mu,\nu)$ one has
        \begin{equation*}
            \scrC_{m,n}(T) = \big(c(\mu) + c(\nu) + N |\nu|\big)\, T\,.
        \end{equation*}
    \end{proposition}
    \begin{proof}
        It suffices to check the assertion on an arbitrarily chosen tensor in each isotypic component. Let $(\mu,\nu)\in \Lambda^{(0)}_{m,n}(N)$, consider a $GL(N)$-invariant subspace of $V_{0}^{m,n}$ isomorphic to $U^{(\mu,\nu)}$ (recall Theorem \bref{thm:tensor_decomposition}). By \cite[Theorem 1.1]{Koike_89}, each embedding of $U^{(\mu,\nu)}$ in $V^{m,n}$ corresponds to a single vector in the irreducible representation of $\mathbb{C}[\Sn{m}\times\Sn{n}]$ labelled by $(\mu,\nu)$. Thus any tensor in the chosen subspace is an eigenvector of $\scrL_{m}$ and $\scrR_{n}$ with the eigenvalue $c(\mu)$ and $c(\nu)$ respectively \cite{Ram_seminormal}. Since $|\nu| = n$, the assertion follows by applying \eqref{eq:JMC} to a traceless tensor (an element of $U^{(\mu,\nu)}$).
        \vskip 2pt

        Now let $(\mu,\nu) \in \Lambda^{(r)}_{m,n}(N)$ for some $r \in \{1,\ldots,\min(m,n)\}$. Consider a $GL(N)$-invariant subspace $W\subseteq V^{m-r,n-r}$ isomorphic to $U^{(\mu,\nu)}$ and apply $\scrC_{m,n}$ to a tensor from $W\otimes E^{\otimes r}$. Note that $E^{i}_{\phantom{i}j}(e_{k}\otimes e^{k}) = 0$, so by Lemma \bref{lem:Casimir} one has
        \begin{equation}
            \scrC_{m,n}(W\otimes E^{\otimes r}) = \scrC_{m-r,n-r}(W)\otimes E^{\otimes r}.
        \end{equation}
        The subspace $W$ is traceless, so the assertion follows by the previous point.
    \end{proof}
    
    With the above result at hand, the problem of determining the spectrum of $\scrA_{m,n}$ reduces to determining the spectrum of $\scrL_{m}$ and $\scrR_{n}$ in $V^{m,n}$ (recall \eqref{eq:Casimir}). For this purpose one invokes the ``see-saw-dual'' Lie group and its centraliser algebra, see the diagram in Section \bref{sec:intro}.

    \paragraph{See-saw-dual pair.}Consider $V^{m,n} = V^{(m,0)}\otimes V^{(0,n)}$ as a representation of $GL(N)\times GL(N)$ where the left and the right components act independently in $V^{\otimes m}$ and $V^{*\otimes n}$: for any $R,S \in GL(N)$ 
        \begin{equation}\label{eq:GLGL_action}
            (R,S)\, (v_{1}\otimes\ldots \otimes v_{m}\otimes \varphi_{1} \otimes\ldots \otimes \varphi_{n}) = R(v_{1})\otimes\ldots \otimes R(v_{m})\otimes S(\varphi_{1}) \otimes\ldots \otimes S(\varphi_{n}).
        \end{equation}
    The centraliser algebra of the action of $GL(N)\times GL(N)$ is the image of $\mathbb{C}[\Sn{m}\times \Sn{n}]$ generated by the transpositions \eqref{eq:transpositions}. Recall that by the classical Schur-Weyl duality (recall Section \bref{sec:intro}), inequivalent irreducible representations of $\mathbb{C}\Sn{s}$ in $V^{\otimes s}$ (as well as in $V^{*\otimes s}$) are indexed by all partitions of $s$ with at most $N$ non-zero components \cite{Weyl}. For any partition $\rho\in\mathcal{P}_{s}(N)$ let $L^{(\rho)}$ denote the corresponding irreducible representation of $\mathbb{C}\Sn{s}$.  Then any irreducible representation of $\mathbb{C}[\Sn{m}\times \Sn{n}]$ in $V^{m,n}$ is equivalent to $L^{(\rho)}\otimes L^{(\sigma)}$, and thus is indexed by a pair of partitions $(\rho,\sigma) \in \mathcal{P}_{m,n}(N)$. 
    \vskip 2pt
        
    As a $(GL(N)\times GL(N), \mathbb{C}[\Sn{m}\times \Sn{n}])$-bimodule $V^{m,n}$ admits the following multiplicity-free decomposition:
        \begin{equation}\label{eq:decomposition_GLGL}
            \def\arraystretch{0.8}
            V^{m,n} \cong \bigoplus_{(\rho,\sigma)\in \mathcal{P}_{m,n}(N)}
            \big(U^{(\rho,\varnothing)}\otimes U^{(\varnothing,\sigma)}\big)\otimes \big(L^{(\rho)}\otimes L^{(\sigma)}\big).
        \end{equation}
    The following result follows (see \cite{Ram_seminormal} for the value of the sum of Jucyc-Murphy elements in an irreducible representation of the symmetric group).
        
    \begin{proposition}\label{prop:eigenvalues_RL}
        For any $m,n \geqslant 1$ and $N\geqslant 1$, let $(\rho,\sigma) \in \mathcal{P}_{m,n}(N)$. Then for any tensor $T\in V^{m,n}$ in the isotypic $GL(N)\times GL(N)$-component $(\rho,\sigma)$ one has
        \begin{equation*}
            \scrL_{m}(T) = c(\rho)\, T\quad \text{and}\quad \scrR_{n}(T) = c(\rho)\, T.
        \end{equation*}
    \end{proposition}
    \noindent With the Propositions \bref{prop:eigenvalues_Casimir} and \bref{prop:eigenvalues_RL} at hand we are in a position to describe $\mathrm{spec}(\scrA_{m,n})$.

    \paragraph{Eigenvalues of $\scrA_{m,n}$.} From \eqref{eq:Casimir} one expresses
    \begin{equation}
        \scrA_{m,n} = \scrL_{m} + \scrR_{n} - \scrC_{m,n} + nN.
    \end{equation}
    The spectra of the operators on the right-hand side of the above expression are given by Propositions \bref{prop:eigenvalues_Casimir} and \bref{prop:eigenvalues_RL}, so one is left to relate the direct-sum decompositions considered therein.
    \vskip 2pt
    
    The decomposition of the restriction to the diagonal subgroup $U^{(\rho,\varnothing)}\otimes U^{(\varnothing,\sigma)}\big\downarrow{}_{GL(N)}$ into irreducible representations $U^{(\mu,\nu)}$ is expressed via irreducible characters as in \eqref{eq:LR_ext_particular} (see \cite[Proposition 2.4 and Theorem 3.6]{Stembridge_87} as well as \cite[Corollary 2.3.1 and Proposition 2.7]{Koike_89}), so that the corresponding multiplicity is
    \begin{equation}\label{eq:branching_GLGL-GL}
        \dim \mathrm{Hom}_{GL(N)}\big(U^{(\mu,\nu)},U^{(\rho,\varnothing)}\otimes U^{(\varnothing,\sigma)}\big) = c^{\mu\nu}_{\rho\sigma}(N).
    \end{equation}
    
    \noindent The set of egenvalues of $\scrA_{m,n}$ is thus summarised via the following theorem.
    
    \begin{theorem}\label{thm:eigenvalues}
        For any $m,n \geqslant 1$ and $N \geqslant 1$, the elements $a \in \mathrm{spec}(\scrA_{m,n})$ are exhausted by integers of the form
        \begin{equation}\label{eq:eigenvalue}
            a = rN + c(\rho\slash\mu) + c(\sigma\slash\nu)
        \end{equation}
        for all $r \in \{0,1,\ldots, \min(m,n)\}$, $(\rho,\sigma) \in \mathcal{P}_{m,n}(N)$ and $(\mu,\nu) \in \Lambda^{(r)}_{m,n}(N)$ such that $c^{\rho\sigma}_{\mu\nu}(N) \neq 0$. In particular, $a = 0$ if and only if $r = 0$.
    \end{theorem}

    \begin{proof}
         First, decompose $V^{m,n}$ into irreducible $GL(N)\times GL(N)$-components $U^{(\rho,\varnothing)}\otimes U^{(\varnothing,\sigma)}$ for all $(\rho,\sigma)\in \mathcal{P}_{m,n}(N)$, and apply Proposition \bref{prop:eigenvalues_RL}. Then restrict each component to the diagonal action of $GL(N)$, so that $U^{(\mu,\nu)}$ with some $(\mu,\nu) \in\Lambda_{m,n}(N)$ occurs in $U^{(\rho,\varnothing)}\otimes U^{(\varnothing,\sigma)}\big\downarrow{}_{GL(N)}$ iff $c^{\rho\sigma}_{\mu\nu}(N) \neq 0$. For each occurrence of $U^{(\mu,\nu)}$ apply Proposition \bref{prop:eigenvalues_Casimir}. 
    \end{proof}

    \noindent Note that the steps of the above proof reproduce the algorithm presented below \eqref{eq:intro_traceless_projector_factorised}.

    \subsection{Further discussion of the construction}\label{sec:projector_further}

    \paragraph{Restricted traceless projectors.} In practice one can be interested in constructing the traceless projection of a tensor with particular permutation symmetries of covariant and contravariant indices. In this case, we show that constructing the traceless projection by applying \eqref{eq:intro_traceless_projector_factorised} requires a smaller number of factors.
    \vskip 2pt
    
    Fix $(\rho,\sigma)\in \mathcal{P}_{m,n}(N)$ and consider a $GL(N)\times GL(N)$-invariant subspace 
    \begin{equation}\label{eq:subset_GLGL}
        W \subset V^{m,n}\quad \text{such that}\quad W \cong U^{(\rho,\varnothing)}\otimes U^{(\varnothing,\sigma)}.
    \end{equation}
    By the classical Schur-Weyl duality, an irreducible representations of $GL(N)$ in $V^{\otimes m}$ (respectively, in $V^{*\otimes n}$) results from applying a primitive idempotent in $\mathbb{C}\Sn{m}$ (respectively, in $\mathbb{C}\Sn{n}$)\footnote{A particular well-known example of primitive idempotents in the symmetric group algebras is given by Young symmetrisers \cite{Weyl}. For the construction of a complete set of primitive orthogonal idempotents see \cite{Ram_seminormal} and references therein.}, so the choice of the above subspace is not unique due to the following multiplicity, recall \eqref{eq:decomposition_GLGL}:
    \begin{equation}
        \dim \mathrm{Hom}_{GL(N)\times GL(N)}(U^{(\rho,\varnothing)}\otimes U^{(\varnothing,\sigma)}, V^{m,n}) = \dim L^{(\rho)}\cdot \dim L^{(\sigma)},
    \end{equation}
    where $\dim L^{(\rho)}$ (respectively, $\dim L^{(\rho)}$) equals the number of standard tableaux of shape $\rho$ (respectively, $\sigma$).
    \vskip 2pt
    
    We are interested in constructing the traceless projection of the subspace $W \subseteq V^{m,n}$ with a smaller number of factors in \eqref{eq:intro_traceless_projector_factorised}. Since each factor is associated with an eigenvalue of $\scrA_{m,n}$ we aim at determining the {\it minimal} subset
    \begin{equation}
        I(\rho,\sigma) \subseteq \mathrm{spec}(\scrA_{m,n})
    \end{equation}
    such that the operator
    \begin{equation}\label{eq:traceless_projector_restricted}
        \scrP^{(\rho,\sigma)}_{m,n} = \prod_{a\in I(\rho,\sigma)\backslash \{0\}} \left(1 - \dfrac{1}{a} \scrA_{m,n}\right)
    \end{equation}
    performs the traceless projection of any subspace \eqref{eq:subset_GLGL}.
    
    \begin{theorem}\label{thm:eigenvalues_restricted}
        For any $m,n\geqslant 1$ and $N\geqslant 1$, for any $(\rho,\sigma) \in \mathcal{P}_{m,n}(N)$ in \eqref{eq:traceless_projector_restricted} the subset $I(\rho,\sigma) \subseteq \mathrm{spec}(\scrA_{m,n})^{\times}$ is constituted by integers of the form
        \begin{equation}\label{eq:eigenvalue_restricted}
            a = Nr + c(\rho\slash\mu) + c(\sigma\slash\nu)
        \end{equation}
    for all $r \in \{0,\ldots,\min(m,n)\}$ and $(\mu,\nu) \in \Lambda^{(r)}_{m,n}(N)$ such that $c_{\rho\sigma}^{\mu\nu}(N) \neq 0$.
    \end{theorem}
    \begin{proof}
        Let $W\subseteq V^{m,n}$ be as in \eqref{eq:subset_GLGL}. Then $U^{(\mu,\nu)}$ with $(\mu,\nu)\in \Lambda_{m,n}(N)$ occurs in $U^{(\rho,\varnothing)}\otimes U^{(\varnothing,\sigma)}\big\downarrow{}_{GL(N)}$ whenever $c_{\rho\sigma}^{\mu\nu}(N) \neq 0$, and restriction of $\scrA_{m,n}$ to any such subspace equals the identity operator times the eigenvalue \eqref{eq:eigenvalue_restricted} (recall the proof of Theorem \bref{thm:eigenvalues}). Note in particular, that neither of the elements of $I(\rho,\sigma)\backslash \{0\}$ can be omitted, so the latter subset is indeed minimal as required.
    \end{proof}

    To construct the set $I(\rho,\sigma)$ one can apply the algorithm below \eqref{eq:intro_traceless_projector_factorised} by reducing the {\it Step 1} to the only pair $(\rho,\sigma)$. As a result, the eigenvalues of $\scrA_{m,n}$ can be organized as follows:
    \begin{equation}
        \mathrm{spec}(\scrA_{m,n}) = \bigcup_{(\rho,\sigma) \in \mathcal{P}_{m,n}(N)} I(\rho,\sigma).
    \end{equation}

    As an immediate corollary of Theorem \bref{thm:eigenvalues_restricted}, it is straightforward to generalise \eqref{eq:traceless_projector_restricted} to a direct sum of $GL(N)\times GL(N)$-invariant subspaces. For any $X \subseteq \mathcal{P}_{m,n}(N)$ denote $\displaystyle I(X) = \bigcup_{(\rho,\sigma) \in X} I(\rho,\sigma)$ and define
    \begin{equation}\label{eq:traceless_projector_restricted_X}
        \scrP_{m,n}^{(X)} = \prod_{a \in I(X)\backslash \{0\}} \left(1 - \dfrac{1}{a} \scrA_{m,n}\right).
    \end{equation}
        
    \begin{corollary}
        For any $m,n \geqslant 1$ and $N\geqslant 1$, for any $X \subseteq \mathcal{P}_{m,n}(N)$ let $W \subseteq V^{m,n}$ be a $GL(N)\times GL(N)$-invariant subspace such that
        \begin{equation}
            W \cong \bigoplus_{(\rho,\sigma) \in X} \big(U^{(\rho,\varnothing)}\otimes U^{(\varnothing,\sigma)} \big)^{\oplus k_{\rho\sigma}}\quad \text{(for some $k_{\rho\sigma} \geqslant 1$)}.
        \end{equation}
        Then the traceless projection $\scrP_{m,n} W \subseteq V^{m,n}_{0}$ results from applying \eqref{eq:traceless_projector_restricted_X}. More to that, none of the factors in \eqref{eq:traceless_projector_restricted_X} can be omitted.
    \end{corollary}
    
    To demonstrate how Theorem \bref{thm:eigenvalues_restricted} applies, let us revisit the example \eqref{eq:traceless_projector_21} of the traceless projector for $V^{2,1}$. By the classical Schur-Weyl duality, the subspace of symmetric tensors $V^{\odot 2} \subseteq V^{\otimes 2}$ is an irreducible representation of $GL(N)$ equivalent to $\ytableausetup{boxsize=5pt} U^{(\ydiagram{2},\varnothing)}$, so Theorem \bref{thm:eigenvalues_restricted} applies to the subspace $V^{\odot 2} \otimes V^{*} \subseteq V^{2,1}$, in which case one has $\ytableausetup{boxsize=5.5pt} \rho = \ydiagram{2}$ and $\ytableausetup{boxsize=5.5pt} \sigma = \ydiagram{1}$. By repeating the same steps as above \eqref{eq:traceless_projector_21} one finds for all $N \geqslant 1$
    \begin{equation}
        \ytableausetup{boxsize=5.5pt} I(\ydiagram{2},\ydiagram{1})\backslash \{0\}= \{N+1\}.
    \end{equation}
    In accordance with \eqref{eq:traceless_projector_21_s}, the traceless projection of $V^{\odot 2} \otimes V^{*}$ can be constructed by applying
    \begin{equation*}
        \ytableausetup{boxsize=4pt}\scrP^{(\ydiagram{2},\ydiagram{1})}_{2,1} = 1 - \dfrac{1}{N+1} \scrA_{2,1}
    \end{equation*}
    instead of \eqref{eq:traceless_projector_21}.
    \vskip 2pt

    Similarly, suppose $N\geqslant 2$ and consider the subspace of anti-symmetric tensors $\bigwedge^{2} V \subseteq V^{\otimes 2}$, which is an irreducible representation of $GL(N)$ equivalent to $\ytableausetup{boxsize=5pt} U^{(\ydiagram{1,1},\varnothing)}$. The subspace $\big(\bigwedge^2 V\big) \otimes V^{*} \subseteq V^{2,1}$ corresponds to the choice $\ytableausetup{boxsize=5.5pt} \rho = \ydiagram{1,1}$ and $\ytableausetup{boxsize=5.5pt} \sigma = \ydiagram{1}$, in which case one has 
    \begin{equation}
        \ytableausetup{boxsize=5.5pt} I(\ydiagram{1,1},\ydiagram{1})\backslash \{0\} = \{N-1\}.
    \end{equation}
    In accordance with \eqref{eq:traceless_projector_21_a}, the traceless projection of $\big(\bigwedge^2 V\big) \otimes V^{*} \subseteq V^{2,1}$ can be constructed by applying
    \begin{equation}\label{eq:traseless_projector_torsion}
        \ytableausetup{boxsize=4pt}\scrP^{(\ydiagram{1,1},\ydiagram{1})}_{2,1} = 1 - \dfrac{1}{N-1} \scrA_{2,1}
    \end{equation}
    instead of \eqref{eq:traceless_projector_21}.
    \vskip 2pt

    In relation to the traceless projection of the Riemann tensor discussed in Section \bref{sec:intro_applications}, consider $m = 3$, $n = 1$ and $N\geqslant 3$, and take $W =  \big( V \wedge V\big) \otimes V\otimes V^{*}$, so that components of any tensor $T \in W$ are anti-symmetric with respect to the transposition of the first two indices: $t^{ijk}{}_{l} = - t^{jik}{}_{l}$. The above subspace is a representation of $GL(N)\times GL(N)$ such that
    \begin{equation}
        \ytableausetup{boxsize=4pt} W \cong \big(U^{(\ydiagram{2,1},\varnothing)} \otimes U^{(\varnothing,\ydiagram{1})}\big) \oplus \big(U^{(\ydiagram{1,1,1},\varnothing)} \otimes U^{(\varnothing,\ydiagram{1})}\big).
    \end{equation}
    Application of the algorithm presented in Section \bref{sec:traceless_projection} is left to the reader as an exercise, so we give directly the result:
    \begin{equation}
        \ytableausetup{boxsize=5.5pt} I(\ydiagram{2,1},\ydiagram{1})\backslash \{0\} = \{N+1,N-1\}\quad \text{and}\quad I(\ydiagram{1,1,1},\ydiagram{1})\backslash \{0\} = \{N-2\}.
    \end{equation}
    The restricted traceless projector \eqref{eq:traceless_projector_X} for $\ytableausetup{boxsize=5.5pt} X = \{(\ydiagram{2,1},\ydiagram{1}),(\ydiagram{1,1,1},\ydiagram{1})\}$ is
    \begin{equation}\label{eq:traceless_projector_Riemann}
        \scrP_{3,1}^{(X)} = \left(1 - \dfrac{1}{N+1} \scrA_{3,1}\right)\left(1 - \dfrac{1}{N-1} \scrA_{3,1}\right)\left(1 - \dfrac{1}{N-2} \scrA_{3,1}\right).
    \end{equation}
    The above operator results from \eqref{eq:intro_traceless_projector_factorised} by omitting the factor with the eigenvalue $N + 2$.

    \paragraph{Branching rules from $C_{m,n}(N)$ to $S_{m,n}(N)$.} As a by-product of the utilised construction, one can describe the restriction of irreducible representations of $C_{m,n}(N)$ to the subalgebra $S_{m,n}(N)$ generated by the transpositions \eqref{eq:transpositions}. Let us first recall some general facts about the irreducible representations of $C_{m,n}(N)$ (recall the key points of Schur-Weyl dualities mentioned in in Section \bref{sec:intro}):
    \begin{itemize}
        \item[1)] The algebra $C_{m,n}(N)$ is semisimple, {\it i.e.} it decomposes as a direct sum of full matrix algebras. Equivalently, $V^{m,n}$ decomposes as a direct sum of irreducible representations of $C_{m,n}(N)$.
        \item[2)] Inequivalent irreducible representations of $C_{m,n}(N)$ are indexed by the set $\Lambda_{m,n}(N)$ as in \eqref{eq:decomposition_GL}.
        \item[3)] Let $(\mu,\nu) \in \Lambda_{m,n}(N)$ and $M_{m,n}^{(\mu,\nu)}$ denote the corresponding irreducible representation of $C_{m,n}(N)$, then one has $d_{\mu\nu} = \dim M_{m,n}^{(\mu,\nu)}$ for the multiplicities in \eqref{eq:decomposition_GL}. 
    \end{itemize}
     More to that, as a $(GL(N),C_{m,n}(N))$-bimodule, $V^{m,n}$ admits the following multiplicity-free decomposition:
    \begin{equation}\label{eq:decomposition_bimodule}
        V^{m,n} \cong \bigoplus_{(\mu,\nu) \in \Lambda_{m,n}(N)} U^{(\mu,\nu)} \otimes M_{m,n}^{(\mu,\nu)}\,.
    \end{equation}
    The above expression manifests the idea that the $GL(N)$-invariant projection onto an irreducible (respectively, isotypic) $GL(N)$-component $(\mu,\nu)$ consists in fixing an element of (respectively, projecting onto) $M_{m,n}^{(\mu,\nu)}$.
    \vskip 2pt
    
    Due to the action of permutations \eqref{eq:transpositions}, any $C_{m,n}(N)$-invariant subspace in $V^{m,n}$ are representations of the group algebra $\mathbb{C}[\Sn{m}\times\Sn{n}]$. Complex representations of finite groups are completely reducible by Maschke's theorem, so for any $(\mu,\nu) \in \Lambda_{m,n}(N)$ we are interested in determining the set of pairs $(\rho,\sigma) \in \mathcal{P}_{m,n}(N)$ such that $L^{(\rho)}\otimes L^{(\sigma)}$ occurs in the decomposition of the restriction of $M_{m,n}^{(\mu,\nu)}$ to the subalgebra $S_{m,n}(N) \subset C_{m,n}(N)$ generated by the action of $\mathbb{C}[\Sn{m}\times \Sn{n}]$ in $V^{m,n}$. With a slight abuse of notation, we denote the latter restriction $M_{m,n}^{(\mu,\nu)}\big\downarrow{}_{\mathbb{C}[\Sn{m}\times\Sn{n}]}$, and write
    \begin{equation}
        \dim\mathrm{Hom}_{\mathbb{C}[\Sn{m}\times\Sn{n}]} \big(L^{(\rho)}\otimes L^{(\sigma)}, M_{m,n}^{(\mu,\nu)}\big)
    \end{equation}
    for the multiplicity of $L^{(\rho)}\otimes L^{(\sigma)}$ in $M_{m,n}^{(\mu,\nu)}\big\downarrow{}_{\mathbb{C}[\Sn{m}\times\Sn{n}]}$. One has the following assertion which follows from the see-saw correspondence described in Section \bref{sec:intro} (recall also \eqref{eq:multiplicity_LR_coef}).
    \begin{proposition}\label{prop:multiplicity_SsubB}
    For any $m,n \geqslant 1$ and $N \geqslant 1$ let $(\rho,\sigma)\in\mathcal{P}_{m,n}(N)$ and $(\mu,\nu)\in \Lambda_{m,n}(N)$. Then
        \begin{equation}\label{eq:multiplicity_SsubB}
        \dim\mathrm{Hom}_{\mathbb{C}[\Sn{m}\times\Sn{n}]} \big(L^{(\rho)}\otimes L^{(\sigma)}, M_{m,n}^{(\mu,\nu)}\big) = c_{\rho\sigma}^{\mu\nu}(N).
    \end{equation}
    \end{proposition}
    \begin{proof}
         The decomposition of the restriction $M^{(\mu,\nu)}_{m,n}\big\downarrow{}_{\mathbb{C}[\Sn{m}\times\Sn{n}]}$ into irreducible components $L^{(\rho)}\otimes L^{(\sigma)}$ follows by comparing \eqref{eq:decomposition_GLGL} with \eqref{eq:decomposition_bimodule} (see, {\it e.g.}, \cite[Theorem 1.7]{Hal96} and references therein) upon applying the branching rules \eqref{eq:branching_GLGL-GL}:
        \begin{equation}\label{eq:multiplicities}
            \dim \mathrm{Hom}_{\mathbb{C}[\Sn{m}\times\Sn{n}]}\big(L^{(\rho)}\otimes L^{(\sigma)}, M^{(\mu,\nu)}_{m,n}\big) = \dim \mathrm{Hom}_{GL(N)}\big(U^{(\mu,\nu)}, U^{(\rho,\varnothing)}\otimes U^{(\varnothing,\sigma)}\big).
        \end{equation}
    \end{proof}

    Let us note that for $N \geqslant m + n$ one has $C_{m,n}(N) \cong B_{m,n}(N)$ \cite[Theorem 5.8]{BCHLLS}, so in this case the multiplicity \eqref{eq:multiplicity_SsubB} can be derived from the representation theory of the walled Brauer algebra \cite[Theorem 3.14]{Hal96} (see also \cite[p. 1492]{Nikitin_07}). This possibility is analysed in detail in the forthcoming subsection. 
    
    \paragraph{Alternative description of $\mathrm{spec}(\scrA_{m,n})$.} By design of the factorised form of the traceless projector \eqref{eq:intro_traceless_projector_factorised}, the number of factors therein can not be reduced unless a subspace of $V^{m,n}$ is considered (for example, as described by Theorem \bref{thm:eigenvalues_restricted}). On the contrary, for any finite set $S \subset \mathbb{C}$ the condition $\mathrm{spec}(\scrA_{m,n})\backslash \{0\} \subseteq S$ is necessary and sufficient to have
    \begin{equation}
        \scrP_{m,n} = \prod_{a\in S\backslash \{0\}} \left(1 - \dfrac{1}{a} \scrA_{m,n}\right).
    \end{equation}
    Sufficiency follows by recalling \eqref{eq:intro_traceless_projector_factorised} and by noting that $\scrA_{m,n}\scrP_{m,n} = 0$, while necessity follows by Lemma \bref{lem:operator_A}: the above formula expresses the projector onto $\mathrm{ker}(\scrA_{m,n})$ thus all non-zero eigenvalues of $\scrA_{m,n}$ are necessarily present in $S$. In what follows we describe a particular set $S = \widetilde{\mathrm{spec}}(\scrA_{m,n})$ which extends $\mathrm{spec}(\scrA_{m,n})$, and which is defined explicitly in terms of the Littlewood-Richardson coefficients.
    \vskip 2pt
    
    The idea is to consider the upper bound for the multiplicities \eqref{eq:multiplicity_SsubB} which follows from the representation theory of $B_{m,n}(N)$. Namely, one recalls that $C_{m,n}(N)$ is a homomorphic image of $B_{m,n}(N)$ (recall the see-saw diagram in Section \bref{sec:intro}), so that an irreducible representation of the former is the image of an indecomposable representation of the latter. Indecomposable representations of $B_{m,n}(N)$ are exhausted by cell modules $\Delta^{(\mu,\nu)}_{m,n}$ indexed by pairs of partitions $(\mu,\nu)$ such that $m - |\mu| = n - |\nu| = r$ for all $r \in \{0,1,\ldots,\min(m,n)\}$. Decomposition of the restriction of the cell modules to the subalgebra $\mathbb{C}[\Sn{m}\times \Sn{n}]$ was given in \cite[Theorem 6.1]{CDDM}: for any $(\rho,\sigma) \in \mathcal{P}_{m,n}$,
    \begin{equation}
        \dim \mathrm{Hom}_{\mathbb{C}[\Sn{m}\times \Sn{n}]} \big(L^{(\rho)}\otimes L^{(\rho)},\Delta^{(\mu,\nu)}_{m,n}\big) = \sum_{\beta \in \mathcal{P}} c^{\rho}_{\mu\beta} c^{\sigma}_{\nu\beta}.
    \end{equation}
    \vskip 2pt
    
    Aside from re-deriving a number of useful properties of the coefficients \eqref{eq:multiplicity_LR_coef}, the above multiplicities provide an alternative description of $\mathrm{spec}(\scrA_{m,n})$ for (infinitely many) special cases for the values of the parameters $m,n,N$. The following lemma extends the known result for $N \geqslant m + n$ when $C_{m,n}(N) \cong B_{m,n}(N)$ (see \cite[Proposition 2.2 and Corollary 2.3.1]{Koike_89}, \cite[Theorem 3.14]{Hal96} and \cite[p. 1492]{Nikitin_07}) to all integers $m,n,N\geqslant 1$. Due to the technical character of the proof we delegate it to Appendix \bref{sec:proof_lemma_estimate}.

    \begin{lemma}\label{lem:estimate_multiplicities} For any $m,n \geqslant 1$ and $N \geqslant 1$ let $(\rho,\sigma)\in \mathcal{P}_{m,n}(N)$ and let $(\mu,\nu)\in \Lambda_{m,n}(N)$, then
        \begin{equation}\label{eq:estimate_multiplicities}
            c^{\mu\nu}_{\rho\sigma}(N) \leqslant \sum_{\beta \in \mathcal{P}(N)} c^{\rho}_{\mu\beta} c^{\sigma}_{\nu\beta}.
        \end{equation}
    The above estimate saturates in the following mutually exclusive cases:
    \begin{itemize}
        \item[1)] $N \geqslant m + n - 1$, or otherwise
        \item[2)] $m = 1$, $n\geqslant 2$, $N \leqslant n - 1$ or $m\geqslant 2$, $n = 1$, $N \leqslant m-1$, or otherwise
        \item[3)] $m,n\geqslant 2$, $N = 1$.
    \end{itemize}
    Otherwise, when $m,n \geqslant 2$ and $2 \leqslant N \leqslant m + n - 2$, there exist $(\rho,\sigma)$ and $(\mu,\nu)$ as above such that the inequality is non-saturated.
    \end{lemma}
    
    For example, to have the two sides of the inequality \eqref{eq:estimate_multiplicities} distinct consider $m = 3$, $n = 2$ and $N = 2$. Take $\ytableausetup{boxsize=5.5pt} (\rho,\sigma) = (\ydiagram{2,1},\ydiagram{1,1})$ and $(\mu,\nu) = \ytableausetup{boxsize=5.5pt} (\ydiagram{2},\ydiagram{1})$, so that $\bar{\sigma} = \varnothing$ and $\mathfrak{s}(\mu,\nu) = \ytableausetup{boxsize=5.5pt} [\ydiagram{3},1]$ (recall \eqref{eq:map_bar} and \eqref{eq:map_staircase}). Then $c^{\lambda}_{\rho\bar{\sigma}} \neq 0$ only for $\ytableausetup{boxsize=5.5pt} \lambda = \ydiagram{2,1}$. In this case $\ytableausetup{boxsize=5.5pt} \big[\,\ydiagram{2,1},1\big] = \big[\ydiagram{1},0\big] \neq \mathfrak{s}(\mu,\nu)$, and therefore
    \begin{equation*}
        \ytableausetup{boxsize=3.5pt}
        c_{\ydiagram{2,1},\ydiagram{1,1}}^{\ydiagram{2},\ydiagram{1}}(2) = 0 \, < \, \sum_{\beta\in \mathcal{P}(2)} c^{\,\ydiagram{2,1}}_{\ydiagram{2},\beta} c^{\,\ydiagram{1,1}}_{\ydiagram{1},\beta} = c^{\,\ydiagram{2,1}}_{\ydiagram{2},\ydiagram{1}} c^{\,\ydiagram{1,1}}_{\ydiagram{1},\ydiagram{1}} = 1.
    \end{equation*}
    \vskip 2pt

    As a related side remark, in \cite[Theorem 2.13(b) and Corollary 2.14(b)]{Hal96} the multiplicities \eqref{eq:multiplicity_SsubB} are given by the right-hand-side of \eqref{eq:estimate_multiplicities} with the only condition that $\ell(\rho)\leqslant N$ and $\ell(\sigma)\leqslant N$. On the other hand, Lemma \bref{lem:estimate_multiplicities} suggests that the latter condition is not sufficient and requires additional restrictions on the values of $m,n,N$ (for example, the condition $N \geqslant m + n$ such that $C_{m,n}(N) \cong B_{m,n}(N)$). Otherwise, the formulations of \cite[Theorem 2.13(b) and Corollary 2.14(b)]{Hal96} may turn out misleading.
    \vskip 2pt

    As a consequence of Lemma \bref{lem:estimate_multiplicities} and Proposition \bref{prop:multiplicity_SsubB} one derives the following properties of the multiplicities \eqref{eq:multiplicity_SsubB} (for the proof, recall also that $c^{\gamma}_{\alpha\beta} \neq 0$ implies $\alpha \subseteq \gamma$ and $\beta \subseteq \gamma$ \cite{Fulton_YT}).
    \begin{corollary}\label{cor:estimate_multiplicities}
        For any $m,n\geqslant 1$ and $N\geqslant 1$, let $(\mu,\nu)\in \Lambda^{(r)}_{m,n}(N)$ for some $r\in \{0,1,\ldots,\min(m,n)\}$. If $L^{(\rho)}\otimes L^{(\sigma)}$ occurs in the restriction $M_{m,n}^{(\mu,\nu)}\big\downarrow{}_{\mathbb{C}[\Sn{m}\times\Sn{n}]}$ then there exists a partition $\beta \in \mathcal{P}(N)$ such that $c^{\rho}_{\mu\beta} c^{\sigma}_{\nu\beta} \neq 0$. In this case
    \begin{equation}
        \mu \subseteq \rho\,,\quad \nu \subseteq \sigma \quad \text{and} \quad |\rho| - |\mu| = |\sigma| - |\nu| = r.
    \end{equation}
    \end{corollary}
    
    Define
    \begin{equation}
    \def\arraystretch{1.4}
    \begin{array}{rl}
        \widetilde{\mathrm{spec}}(\scrA_{m,n}) = & \big\{Nr + c(\rho\slash\mu) + c(\sigma\slash\nu) \;:\; r \in \{0,1,\ldots,\min(m,n)\},\,(\mu,\nu) \in \Lambda^{(r)}_{m,n}(N), \\
        \hfill & \;(\rho,\sigma)\in \mathcal{P}_{m,n}(N)\;\text{such that}\; c^{\rho}_{\mu\beta}c^{\sigma}_{\nu\beta} \neq 0\;\text{for some}\;\beta\subseteq \rho\cap\sigma\} \cap \mathbb{N}_{0}
    \end{array}
    \end{equation}
    The requirement for the elements of the above set to be non-negative is imposed in view of Lemma \bref{lem:operator_A}, so by Corollary \bref{cor:estimate_multiplicities} one indeed has $\mathrm{spec}(\scrA_{m,n}) \subseteq \widetilde{\mathrm{spec}}(\scrA_{m,n})$.

    \begin{proposition}\label{prop:spec_A_alternative}
        Let $m,n \geqslant 1$ and $N\geqslant 1$ be as in Lemma \bref{lem:estimate_multiplicities} such that the estimate \eqref{eq:estimate_multiplicities} saturates, then
        \begin{equation}\label{eq:alternative_spec_A}
            \mathrm{spec}(\scrA_{m,n}) = \widetilde{\mathrm{spec}}(\scrA_{m,n}).
        \end{equation}
    \end{proposition}
    
    A separate interesting problem is to identify the necessary condition for \eqref{eq:alternative_spec_A}. The requirement of non-negativity of the eigenvalues due to the point $(3)$ of Lemma \bref{lem:operator_A} was already taken into account in the definition of $\widetilde{\mathrm{spec}}(\scrA_{m,n})$ by hand. To show that the latter requirement becomes relevant when $m,n,N$ are such that the bound \eqref{eq:estimate_multiplicities} is non-saturated, consider $m = n = N = 4$ and take $\rho = \sigma = (1_{4})$. Then for $\mu = \nu = (1_{2})$ the right-hand side of \eqref{eq:estimate_multiplicities} is non-zero, on the other hand $Nr + c(\rho\slash\mu) + c(\sigma\slash\nu) = -2$.
    \vskip 2pt
    
    Continuing the above example, there is another possibility $\mu = \nu = (1)$ such that the right-hand side of \eqref{eq:estimate_multiplicities} is non-zero. In this case one has $Nr + c(\rho\slash\mu) + c(\sigma\slash\nu) = 0$ for $r \neq 0$, and thus $c^{\mu\nu}_{\rho\sigma}(4) = 0$ since otherwise there would be a contradiction with Theorem \bref{thm:eigenvalues}.
    \vskip 2pt
    
    Apart from excluding negative elements, the rest of the necessary condition does not manifest itself in the simplest cases with $m,n\geqslant 2$ and $2 \leqslant N \leqslant m+n-2$: so far the author did not succeed in finding neither an example such that $\widetilde{\mathrm{spec}}(\scrA_{m,n}) \not\subseteq \mathrm{spec}(\scrA_{m,n})$ nor a convincing argument to expect \eqref{eq:alternative_spec_A} to hold for all $m,n,N \geqslant 1$.
    \vskip 2pt 

    \section{Splitting idempotents}
    \subsection{Splitting idempotent in \texorpdfstring{$C_{m,n}(N)$}{the centraliser algebra}}\label{sec:splitting_idempotent_C}
    
    \paragraph{Traceless projector as a splitting idempotent.} In Section \bref{sec:traceless_projection}, the traceless projector \eqref{eq:traceless_projector} was identified via the choice of the complement  $V_{1}^{m,n}$ of the traceless subspace. Then, by virtue of Theorem \bref{thm:tensor_decomposition} one concluded that $\scrP_{m,n}$ was a central idempotent in $C_{m,n}(N)$. In this section we go other way around, and describe an equivalent way of defining the traceless projector $\scrP_{m,n}$ as a particular central idempotent in the centraliser algebra $C_{m,n}(N)$.
    \vskip 2pt
    
    As a direct consequence of Theorem \bref{thm:cenraliser_generate}, the traceless subspace $V_{0}^{m,n}$ is $C_{m,n}(N)$-invariant, so its annihilator in $C_{m,n}(N)$ is a two-sided ideal which we denote $\mathcal{J} \subset C_{m,n}(N)$. There is a natural surjective homomorphism from $C_{m,n}(N)$ onto the the quotient algebra $C_{m,n}(N)\slash \mathcal{J}$, which is expressed via the short exact sequence of algebras \eqref{eq:intro_exact_sequence_C}. As explained in Section \bref{sec:intro}, the latter exact sequence splits by means of a uniquely-defined central idempotent in $C_{m,n}(N)$.

    \begin{theorem}\label{thm:splitting_idempotent_C}
        The traceless projector \eqref{eq:traceless_projector} is the splitting idempotent of the short exact sequence \eqref{eq:intro_exact_sequence_C}:
        \begin{equation*}
            C_{m,n}(N) \cong \scrP_{m,n}C_{m,n}(N) \oplus \mathcal{J} \quad\text{(direct sum of algebras).}
        \end{equation*}
    \end{theorem}
    \begin{proof}
         Recall that the traceless projector \eqref{eq:traceless_projector} is central by Proposition \bref{prop:P_central}. Thus, to prove the assertion we show that: {\it (i)} the two ideals $\mathcal{J}$ and $\scrP_{m,n} C_{m,n}(N)$ are complementary in $C_{m,n}(N)$ (and thus annihilate each other), and {\it (ii)} $\scrP_{m,n} C_{m,n}(N) \cong C_{m,n}(N) \slash \mathcal{J}$.
        \vskip 2pt

        For the first point, let $x\in \scrP_{m,n} C_{m,n}(N)$, then one has $x = \scrP_{m,n} x = x\scrP_{m,n}$. If also $x \in \mathcal{J}$ then $x\scrP_{m,n} = 0$ since $\mathcal{J}$ annihilates the image of $\scrP_{m,n}$.
        \vskip 2pt

        To prove the second point, for any $x\in C_{m,n}(N)$ let $[x] \in C_{m,n}(N)\slash \mathcal{J}$ denote the corresponding equivalence class modulo $\mathcal{J}$ and consider the map
        \begin{equation}
        \def\arraystretch{1.4}
        \begin{array}{rccc}
            \varphi \,: & C_{m,n}(N)\slash \mathcal{J} & \to &  \scrP_{m,n} C_{m,n}(N) \\
            \hfill & [x] & \mapsto & \scrP_{m,n} x
        \end{array}
        \end{equation}
        The map $\varphi$ is well-defined since any two representatives differ by an element of $\mathcal{J}$ which annihilates the image of $\scrP_{m,n}$. Clearly, $\varphi$ is a surjective homomorphism of algebras. To prove that $\varphi$ is injective suppose that $\scrP_{m,n} x_{1} = \scrP_{m,n} x_{2}$ for some $x_{1},x_{2} \in C_{m,n}(N)$. Then $(x_1 - x_2) \scrP_{m,n} = 0$ which implies $x_{1} - x_{2} \in \mathcal{J}$ because $\mathrm{Im}(\scrP_{m,n}) = V_{0}^{m,n}$.
    \end{proof}

    \paragraph{Complementary ideal $\mathcal{I}$.} Let us consider the ideal $\mathcal{I} = \scrP_{m,n}C_{m,n}(N)$ in \eqref{eq:intro_exact_sequence_C} in a more detail. When $N \geqslant m+n$, one has $
    \mathcal{I} \cong \mathbb{C}[\Sn{m}\times \Sn{n}]$, which follows from the isomorphism $C_{m,n}(N) \cong B_{m,n}(N)$ \cite[Theorem 5.8]{BCHLLS} together with Theorem \bref{thm:wB_splitting_idempotent}. In what follows we prove the aforementioned isomorphism by analysing the traceless projection, and thus without referring to the representation theory of the walled Brauer algebra $B_{m,n}(N)$. On the other hand, for $N < m+n$ we show that $\mathcal{I}$ is not isomorphic to $\mathbb{C}[\Sn{m}\times \Sn{n}]$, but rather to a quotient thereof.
    \vskip 2pt
    
    For example, for $N = 2$ consider $V^{2,1}$ which has the following $\big(GL(2)\times GL(2),\mathbb{C}[\Sn{2}\times \Sn{1}]\big)$-bimodule decomposition (recall \eqref{eq:decomposition_GLGL}):
    \begin{equation}\label{eq:example_S_in_C}
    \ytableausetup{boxsize=4pt} 
        V^{2,1} \cong \big(U^{(\ydiagram{2},\varnothing)}\otimes U^{(\varnothing,\ydiagram{1})}\big)\otimes \big(L^{(\ydiagram{2})}\otimes L^{(\ydiagram{1})} \big) \,\oplus\, \big(U^{(\ydiagram{1,1},\varnothing)}\otimes U^{(\varnothing,\ydiagram{1})} \big)\otimes \big(L^{(\ydiagram{1,1})}\otimes L^{(\ydiagram{1})} \big).
    \end{equation}
    Upon restriction to the diagonal subgroup $GL(2)$ one has
    \begin{equation}
    \ytableausetup{boxsize=4pt} 
        \big(U^{(\ydiagram{2},\varnothing)}\otimes U^{(\varnothing,\ydiagram{1})}\big)\big\downarrow{}_{GL(N)} \cong U^{(\ydiagram{2},\ydiagram{1})} \oplus U^{(\ydiagram{1},\varnothing)}\quad \text{and}\quad \big(U^{(\ydiagram{1,1},\varnothing)} \otimes U^{(\varnothing,\ydiagram{1})}\big)\big\downarrow{}_{GL(N)}\cong U^{(\ydiagram{1},\varnothing)},
    \end{equation}
    which means that only the former component in \eqref{eq:example_S_in_C} admits a non-trivial traceless projection (note that $\ytableausetup{boxsize=5.5pt} \ell(\ydiagram{1,1}) + \ell(\ydiagram{1}) > 2$ and recall Theorem \bref{thm:tensor_decomposition}). As a result, one has $1 - \tau_{12} \in \mathcal{J}$, so the two-sided ideal generated by the latter combination is factored out from the image of $\mathbb{C}[\Sn{2}\times \Sn{1}]$ in $C_{2,1}(2)$.
    \vskip 2pt

    The idea explained in the above example allows us to prove the following assertion.
    \begin{proposition}\label{prop:image_Sn}
    For any $m,n\geqslant 1$ and $N\geqslant 1$,
    \begin{equation}
        \mathcal{I} \cong \mathbb{C}[\Sn{m}\times \Sn{n}]\quad \text{if and only if} \quad N \geqslant m+n.
    \end{equation}
    \end{proposition}
    \begin{proof}
        Recall that $S_{m,n}(N) \subset C_{m,n}(N)$ denotes the image of $\mathbb{C}[\Sn{m}\times \Sn{n}]$. Let us first prove that for any $N\geqslant m+n$ one has the following vector-space decomposition of $C_{m,n}(N)$:
        \begin{equation}
            C_{m,n}(N) = S_{m,n}(N) \oplus \mathcal{J}.
        \end{equation}
        Indeed, let $f\in \mathcal{J}\cap S_{m,n}(N)$, then $f(V_{0}^{m,n}) = 0$ and at the same time $f$ results from an element $z\in \mathbb{C}[\Sn{m}\times \Sn{n}]$. We claim that $z$ annihilates all irreducible representations of $\mathbb{C}[\Sn{m}\times \Sn{n}]$. Indeed, note that for all $(\rho,\sigma) \in \mathcal{P}_{m,n}(N)$ one has $\ell(\rho) + \ell(\sigma) \leqslant N$ and recall \eqref{eq:decomposition_bimodule}: on one hand all $U^{(\rho,\sigma)} \neq \{0\}$, while on the other hand $M^{(\rho,\sigma)}_{m,n} \cong L^{(\rho)}\otimes L^{(\sigma)}$ by \cite[Theorem 1.1]{Koike_89}. As a result, $z$ belongs to the Jacobson radical of $\mathbb{C}[\Sn{m}\times \Sn{n}]$, but the latter is trivial since $\mathbb{C}[\Sn{m}\times \Sn{n}]$ is semisimple by Mashke's theorem.
        \vskip 2pt

        Now let $N < m + n$. There exists $(\rho,\sigma) \in \mathcal{P}_{m,n}(N)$ such that $\ell(\rho) + \ell(\sigma) > N$, so that $U^{(\rho,\sigma)} = \{0\}$. Take $Z_{\rho} \in \mathbb{C}\Sn{m}$ and $Z_{\sigma} \in \mathbb{C}\Sn{n}$ two central idempotents and apply $Z_{\rho}\otimes Z_{\sigma}\in \mathbb{C}[\Sn{m}\times\Sn{n}]$ to $V^{m,n}$. Here for any two permutations $u\in \Sn{m}$ and $v\in \Sn{n}$ one defines
        \begin{equation}\label{eq:otimes_SnSn}
            u\otimes v \in \Sn{m}\times \Sn{n}
        \end{equation}
        such that $u$ (respectively, $v$) permutes the $m$ leftmost (respectively, $n$ rightmost) elements, and extends the so-defined map to $\mathbb{C}[\Sn{m}\times \Sn{n}]$ by bilinearity. The resulting non-trivial $GL(N)\times GL(N)$-invariant subspace is traceless, so the image of $Z_{\rho}\otimes Z_{\sigma}$ in $C_{m,n}(N)$ is a non-zero element in $\mathcal{J}$ which is factored out from the image of $\mathbb{C}[\Sn{m}\times\Sn{n}]$ in $C_{m,n}(N)\slash\mathcal{J}$. 
    \end{proof}

    \subsection{Splitting idempotent in the walled Brauer algebra}\label{sec:wB_algebra}
    
    \paragraph{Walled Brauer algebra.} An efficient way of working with the centraliser algebra $C_{m,n}(N)$ consists in addressing its elements via a particular associative diagram algebra $B_{m,n}(N)$ referred to as {\it walled Brauer algebra}. In the sequel the parameter $N$ is allowed to take any complex value. The algebra $B_{m,n}(\delta)$ for $\delta\in \mathbb{C}$ was introduced in \cite{BCHLLS} as a particular subalgebra of the Brauer algebra $B_{m+n}(\delta)$ \cite{Brauer}, and independently in the study of link invariants in \cite{Turaev_1990}.
    \vskip 2pt
    
    The basis of $B_{m,n}(\delta)$ is constituted by {\it walled diagrams} introduced as follows. Consider two sets of $m+n$ horizontally aligned points (nodes) on the plane, forming $m+n$ vertically aligned pairs. The $m$ leftmost pairs are separated from the $n$ rightmost ones by a vertical dashed line referred to as {\it wall}. Each node is connected to exactly one other node by a convex line lying within the rectangle separated by the leftmost and the rightmost upper and lower nodes according to the following rule: the endpoints of a line belong to the same row if and only if the line crosses the wall. Any line crossing the wall is referred to as {\it arc}, and {\it passing line} otherwise. It is straightforward to see that if a diagram has $r\in \{0,1,\ldots,\min(m,n)\}$ arcs with endpoints in the upper row, then there are exactly $r$ arcs with endpoints in the lower row, and {\it vice versa}. In this respect, a diagram is said to have $r$ arcs if there are $r$ arcs with endpoints either in the upper or in the lower row. For example, consider the following two walled diagrams in $B_{4,3}(\delta)$ with two arcs each:
    \begin{equation}\label{eq:wB_example_diagrams}
    b_1 = \;\; \begin{tikzpicture}[baseline=(current bounding box.center)]
            \def \dx {0.7};
            \def \dy {0.5};
            \draw (0,\dy) -- (\dx*2,-\dy);
            \draw (\dx*2,\dy) -- (\dx,-\dy);
            \draw (\dx*6,\dy) -- (5*\dx,-\dy);
            \draw (\dx,\dy) .. controls (1.8*\dx,0.3*\dy) and (3.2*\dx,0.3*\dy) .. (4*\dx,\dy);
            \draw (3*\dx,\dy) .. controls (3.5*\dx,0.3*\dy) and (4.5*\dx,0.3*\dy) .. (5*\dx,\dy);
            \draw (0*\dx,-\dy) .. controls (1.5*\dx,-0.3*\dy) and (4.5*\dx,-0.3*\dy) .. (6*\dx,-\dy);
            \draw (3*\dx,-\dy) .. controls (3.2*\dx,-0.7*\dy) and (3.8*\dx,-0.7*\dy) .. (4*\dx,-\dy);
            \foreach \i in {0,...,6}
            \foreach \j in {-1,1}
            {\filldraw[white] (\dx*\i,\dy*\j) circle (1pt);
            \draw (\dx*\i,\dy*\j) circle (1pt);}
            \draw[dashed] (\dx*3.5,1.4*\dy) -- (\dx*3.5,-1.4*\dy);
        \end{tikzpicture}
        \quad\quad \text{and}\quad\quad
        b_2 = \;\; \begin{tikzpicture}[baseline=(current bounding box.center)]
            \def \dx {0.7};
            \def \dy {0.5};
            \draw (0,\dy) -- (0,-\dy);
            \draw (\dx,\dy) -- (2*\dx,-\dy);
            \draw (5*\dx,\dy) -- (6*\dx,-\dy);
            \draw (2*\dx,\dy) .. controls (2.5*\dx,0.3*\dy) and (3.5*\dx,0.3*\dy) .. (4*\dx,\dy);
            \draw (3*\dx,\dy) .. controls (3.5*\dx,0.3*\dy) and (5.5*\dx,0.3*\dy) .. (6*\dx,\dy);
            \draw (1*\dx,-\dy) .. controls (1.5*\dx,-0.3*\dy) and (4.5*\dx,-0.3*\dy) .. (5*\dx,-\dy);
            \draw (3*\dx,-\dy) .. controls (3.2*\dx,-0.7*\dy) and (3.8*\dx,-0.7*\dy) .. (4*\dx,-\dy);
            \foreach \i in {0,...,6}
            \foreach \j in {-1,1}
            {\filldraw[white] (\dx*\i,\dy*\j) circle (1pt);
            \draw (\dx*\i,\dy*\j) circle (1pt);}
            \draw[dashed] (\dx*3.5,1.4*\dy) -- (\dx*3.5,-1.4*\dy);
        \end{tikzpicture}
    \end{equation}
    The number of walled diagrams equals $(m+n)!$ and clearly does not depend on a particular value of the parameter $\delta$ \cite{BCHLLS}. A particularly simple proof of this fact is given in \cite{Nikitin_07} by establishing a bijection between walled diagrams and permutations of $m+n$ elements.
    \vskip 2pt

    Let $b^{\prime},b^{\prime\prime} \in B_{m,n}(N)$, the product $b^{\prime}b^{\prime\prime}$ of two walled diagrams is performed as follows: place $b^{\prime\prime}$ above $b^{\prime}$ and identify the lower nodes of the former with the upper nodes of the latter. Let $l$ be the number of closed loops, then $b^{\prime}b^{\prime\prime}$ is a walled diagram obtained by omitting the loops, straightening the resting lines and multiplying the result by $\delta^{l}$. For example, the product of the two walled diagrams in \eqref{eq:wB_example_diagrams} reads as follows:
    \begin{equation}\label{eq:B_product}
        b_{1} b_{2} = \;\;
        \begin{tikzpicture}[baseline=(current bounding box.center)]
            \def \dx {0.7};
            \def \dy {0.5};
            \draw (0,0) -- (\dx*2,-\dy);
            \draw (\dx*2,0) -- (\dx,-\dy);
            \draw (\dx*6,0) -- (5*\dx,-\dy);
            \draw (\dx,0) .. controls (1.8*\dx,-0.4*\dy) and (3.2*\dx,-0.4*\dy) .. (4*\dx,0);
            \draw (3*\dx,0) .. controls (3.5*\dx,-0.4*\dy) and (4.5*\dx,-0.4*\dy) .. (5*\dx,0);
            \draw (0*\dx,-\dy) .. controls (1.5*\dx,-0.3*\dy) and (4.5*\dx,-0.3*\dy) .. (6*\dx,-\dy);
            \draw (3*\dx,-\dy) .. controls (3.2*\dx,-0.7*\dy) and (3.8*\dx,-0.7*\dy) .. (4*\dx,-\dy);
            \draw (0,\dy) -- (0,0);
            \draw (\dx,\dy) -- (2*\dx,0);
            \draw (5*\dx,\dy) -- (6*\dx,0);
            \draw (2*\dx,\dy) .. controls (2.5*\dx,0.3*\dy) and (3.5*\dx,0.3*\dy) .. (4*\dx,\dy);
            \draw (3*\dx,\dy) .. controls (3.5*\dx,0.3*\dy) and (5.5*\dx,0.3*\dy) .. (6*\dx,\dy);
            \draw (1*\dx,0) .. controls (1.5*\dx,0.5*\dy) and (4.5*\dx,0.5*\dy) .. (5*\dx,0);
            \draw (3*\dx,0) .. controls (3.2*\dx,0.3*\dy) and (3.8*\dx,0.3*\dy) .. (4*\dx,0);
            \foreach \i in {0,...,6}
            \foreach \j in {-1,0,1}
            {\filldraw[white] (\dx*\i,\dy*\j) circle (1pt);
            \draw (\dx*\i,\dy*\j) circle (1pt);}
            \draw[dashed] (\dx*3.5,1.4*\dy) -- (\dx*3.5,-1.4*\dy);
        \end{tikzpicture}
        \;\; = \;\;
        \delta \;\;\begin{tikzpicture}[baseline=(current bounding box.center)]
            \def \dx {0.7};
            \def \dy {0.5};
            \draw (0,\dy) -- (\dx*2,-\dy);
            \draw (\dx,\dy) -- (\dx,-\dy);
            \draw (5*\dx,\dy) -- (5*\dx,-\dy);
            \draw (2*\dx,\dy) .. controls (2.5*\dx,0.3*\dy) and (3.5*\dx,0.3*\dy) .. (4*\dx,\dy);
            \draw (3*\dx,\dy) .. controls (3.5*\dx,0.3*\dy) and (5.5*\dx,0.3*\dy) .. (6*\dx,\dy);
            \draw (0*\dx,-\dy) .. controls (1.5*\dx,-0.3*\dy) and (4.5*\dx,-0.3*\dy) .. (6*\dx,-\dy);
            \draw (3*\dx,-\dy) .. controls (3.2*\dx,-0.7*\dy) and (3.8*\dx,-0.7*\dy) .. (4*\dx,-\dy);
            \foreach \i in {0,...,6}
            \foreach \j in {-1,1}
            {\filldraw[white] (\dx*\i,\dy*\j) circle (1pt);
            \draw (\dx*\i,\dy*\j) circle (1pt);}
            \draw[dashed] (\dx*3.5,1.4*\dy) -- (\dx*3.5,-1.4*\dy);
        \end{tikzpicture}
    \end{equation}

    \paragraph{Generating set.} Consider the following three types of elements of $B_{m,n}(\delta)$ defined for all $1 \leqslant \sfa,\sfb \leqslant m$ and $1^{\prime} \leqslant \hsfa,\hsfb \leqslant n^{\prime}$:
    \begin{equation}\label{eq:generators_wB}
    \def\arraystretch{1.7}
    \begin{array}{c}
        t_{\sfa\sfb} = \;\begin{tikzpicture}[baseline=(current bounding box.center)]
            \def \dx {0.4};
            \def \dy {0.5};
            \draw (0,\dy) -- (0,-\dy);
            \draw (2*\dx,\dy) -- (2*\dx,-\dy);
            \draw (4*\dx,\dy) -- (4*\dx,-\dy);
            \draw (3*\dx,\dy) -- (7*\dx,-\dy);
            \draw (7*\dx,\dy) -- (3*\dx,-\dy);
            \draw (6*\dx,\dy) -- (6*\dx,-\dy);
            \draw (8*\dx,\dy) -- (8*\dx,-\dy);
            \draw (10*\dx,\dy) -- (10*\dx,-\dy);
            \draw (11*\dx,\dy) -- (11*\dx,-\dy);
            \draw (13*\dx,\dy) -- (13*\dx,-\dy);
            \node at (3*\dx, -1.6*\dy) {\scalebox{0.8}{$\vphantom{1^{\prime}}\sfa$}};
            \node at (7*\dx, -1.6*\dy) {\scalebox{0.8}{$\vphantom{1^{\prime}}\sfb$}};
            \node at (0, -1.6*\dy) {\scalebox{0.8}{$\vphantom{1^{\prime}} 1$}};
            \node at (10*\dx, -1.6*\dy) {\scalebox{0.8}{$\vphantom{1^{\prime}} m$}};
            \node at (11*\dx, -1.6*\dy) {\scalebox{0.8}{$1^{\prime}$}};
            \node at (13*\dx, -1.6*\dy) {\scalebox{0.8}{$\vphantom{1^{\prime}} n^{\prime}$}};
            \foreach \i in {0,2,3,4,6,7,8,10,11,13}
            \foreach \j in {-1,1}
            {\filldraw[white] (\dx*\i,\dy*\j) circle (1pt);
            \draw (\dx*\i,\dy*\j) circle (1pt);}
            \draw[dashed] (\dx*10.5,1.4*\dy) -- (\dx*10.5,-1.4*\dy);
            \foreach \i in {1,5,9,12}
            \foreach \j in {-1,1}
            {\filldraw[black] (\dx*\i,\dy*\j) circle (0.3pt);
            \filldraw[black] (\dx*\i - 0.3*\dx,\dy*\j) circle (0.3pt);
            \filldraw[black] (\dx*\i + 0.3*\dx,\dy*\j) circle (0.3pt);}
        \end{tikzpicture}
        \quad\quad\quad t_{\hsfa\hsfb} = \;\begin{tikzpicture}[baseline=(current bounding box.center)]
            \def \dx {0.4};
            \def \dy {0.5};
            \draw (0,\dy) -- (0,-\dy);
            \draw (-2*\dx,\dy) -- (-2*\dx,-\dy);
            \draw (-4*\dx,\dy) -- (-4*\dx,-\dy);
            \draw (-3*\dx,\dy) -- (-7*\dx,-\dy);
            \draw (-7*\dx,\dy) -- (-3*\dx,-\dy);
            \draw (-6*\dx,\dy) -- (-6*\dx,-\dy);
            \draw (-8*\dx,\dy) -- (-8*\dx,-\dy);
            \draw (-10*\dx,\dy) -- (-10*\dx,-\dy);
            \draw (-11*\dx,\dy) -- (-11*\dx,-\dy);
            \draw (-13*\dx,\dy) -- (-13*\dx,-\dy);
            \node at (-3*\dx, -1.6*\dy) {\scalebox{0.8}{$\vphantom{1^{\prime}}\hsfb$}};
            \node at (-7*\dx, -1.6*\dy) {\scalebox{0.8}{$\vphantom{1^{\prime}}\hsfa$}};
            \node at (0, -1.6*\dy) {\scalebox{0.8}{$\vphantom{1^{\prime}} n^{\prime}$}};
            \node at (-10*\dx, -1.6*\dy) {\scalebox{0.8}{$\vphantom{1^{\prime}} 1^{\prime}$}};
            \node at (-11*\dx, -1.6*\dy) {\scalebox{0.8}{$\vphantom{1^{\prime}} m$}};
            \node at (-13*\dx, -1.6*\dy) {\scalebox{0.8}{$\vphantom{1^{\prime}} 1$}};
            \foreach \i in {0,-2,-3,-4,-6,-7,-8,-10,-11,-13}
            \foreach \j in {-1,1}
            {\filldraw[white] (\dx*\i,\dy*\j) circle (1pt);
            \draw (\dx*\i,\dy*\j) circle (1pt);}
            \draw[dashed] (-\dx*10.5,1.4*\dy) -- (-\dx*10.5,-1.4*\dy);
            \foreach \i in {-1,-5,-9,-12}
            \foreach \j in {-1,1}
            {\filldraw[black] (\dx*\i,\dy*\j) circle (0.3pt);
            \filldraw[black] (\dx*\i - 0.3*\dx,\dy*\j) circle (0.3pt);
            \filldraw[black] (\dx*\i + 0.3*\dx,\dy*\j) circle (0.3pt);}
        \end{tikzpicture}\\
        t_{\sfa\hsfb} = \;\begin{tikzpicture}[baseline=(current bounding box.center)]
            \def \dx {0.4};
            \def \dy {0.5};
            \draw (0,\dy) -- (0,-\dy);
            \draw (2*\dx,\dy) -- (2*\dx,-\dy);
            \draw (4*\dx,\dy) -- (4*\dx,-\dy);
            \draw (3*\dx,\dy) .. controls (4.5*\dx,0.3*\dy) and (8.5*\dx,0.3*\dy) .. (10*\dx,\dy);
            \draw (3*\dx,-\dy) .. controls (4.5*\dx,-0.3*\dy) and (8.5*\dx,-0.3*\dy) .. (10*\dx,-\dy);
            \draw (6*\dx,\dy) -- (6*\dx,-\dy);
            \draw (7*\dx,\dy) -- (7*\dx,-\dy);
            \draw (9*\dx,\dy) -- (9*\dx,-\dy);
            \draw (11*\dx,\dy) -- (11*\dx,-\dy);
            \draw (13*\dx,\dy) -- (13*\dx,-\dy);
            \node at (3*\dx, -1.6*\dy) {\scalebox{0.8}{$\vphantom{1^{\prime}}\sfa$}};
            \node at (10*\dx, -1.6*\dy) {\scalebox{0.8}{$\vphantom{1^{\prime}}\hsfb$}};
            \node at (0, -1.6*\dy) {\scalebox{0.8}{$\vphantom{1^{\prime}} 1$}};
            \node at (6*\dx, -1.6*\dy) {\scalebox{0.8}{$\vphantom{1^{\prime}} m$}};
            \node at (7*\dx, -1.6*\dy) {\scalebox{0.8}{$1^{\prime}$}};
            \node at (13*\dx, -1.6*\dy) {\scalebox{0.8}{$\vphantom{1^{\prime}} n^{\prime}$}};
            \foreach \i in {0,2,3,4,6,7,9,10,11,13}
            \foreach \j in {-1,1}
            {\filldraw[white] (\dx*\i,\dy*\j) circle (1pt);
            \draw (\dx*\i,\dy*\j) circle (1pt);}
            \draw[dashed] (\dx*6.5,1.4*\dy) -- (\dx*6.5,-1.4*\dy);
            \foreach \i in {1,5,8,12}
            \foreach \j in {-1,1}
            {\filldraw[black] (\dx*\i,\dy*\j) circle (0.3pt);
            \filldraw[black] (\dx*\i - 0.3*\dx,\dy*\j) circle (0.3pt);
            \filldraw[black] (\dx*\i + 0.3*\dx,\dy*\j) circle (0.3pt);}
        \end{tikzpicture}
    \end{array}
    \end{equation} 
    
    \noindent The following result is well-known \cite{BCHLLS}.
    \begin{proposition}\label{prop:wB_generate}
        For any $m,n \geqslant 1$ and $\delta\in \mathbb{C}$ the elements \eqref{eq:generators_wB} generate the algebra $B_{m,n}(\delta)$.
    \end{proposition}
    
    \paragraph{Action on mixed tensors.} For any $m,n\geqslant 1$ and any integer $N\geqslant 1$ the elements of $B_{m,n}(N)$ act on $V^{m,n}$ via the surjective homomorphism:
    \begin{equation}\label{eq:hom_Br}
        \mathfrak{b} : B_{m,n}(N) \to C_{m,n}(N).
    \end{equation}
    By Proposition \bref{prop:wB_generate}, to define \eqref{eq:hom_Br} it is sufficient to give the $\mathfrak{b}$-image of the generating elements \eqref{eq:generators_wB}:
    \begin{equation}
        \mathfrak{b}(t_{\sfa\sfb}) = \tau_{\sfa\sfb}\,,\quad \mathfrak{b}(t_{\hsfa\hsfb}) = \tau_{\hsfa\hsfb}\quad \text{and}\quad \mathfrak{b}(t_{\sfa\hsfb}) = \tau_{\sfa\hsfb}\,.
    \end{equation}
    The action of a walled diagram in $V^{m,n}$ is conveniently summarised via the following rule:
    \begin{itemize}
        \item[1)] Label the $m$ leftmost (respectively, $n$ rightmost) nodes in the upper row by the $m$ factors $V$ (respectively, $n$ factors $V^{*}$) in $V^{m,n}$.
        \item[2)] For each two nodes in the upper row joined by an arc, apply contraction to the corresponding pair $V,V^{*}$. 
        \item[3)] For the remaining upper nodes, permute the corresponding factors $V$ and $V^{*}$ by following the passing lines downwards.
        \item[4)] Insert the invariant \eqref{eq:cotrace} at positions occupied by endpoints of the arcs in the lower row.
    \end{itemize}
    It is also useful to describe how the above rule translates to tensor components $t^{i_{1}\ldots i_{m}}{}_{j_{1}\ldots j_{n}}$ in \eqref{eq:mixed_basis}:
    \begin{itemize}
        \item[1)] Label the $m$ leftmost (respectively, $n$ rightmost) nodes in the lower row by $i_{1},\ldots, i_{m}$ (respectively, by $j_{1},\ldots, j_{n}$) from left to right.
        \item[2)] For each two nodes at positions $1 \leqslant \sfa\leqslant m$ and $1 \leqslant \hsfb\leqslant n$ in the upper row (respectively, in the lower row) joined by an arc, contract the corresponding pair of indices, $t^{\ldots i_{\sfa}\ldots}{}_{\ldots j_{\hsfb}\ldots} \mapsto t^{\ldots k\ldots}{}_{\ldots k\ldots}$ (respectively, multiply the components by the Kronecker delta, $t^{\ldots}{}_{\ldots} \mapsto \delta^{i_{\sfa}}_{j_{\hsfb}}t^{\ldots}{}_{\ldots}$).
        \item[3)] Label the remaining nodes in the upper row by the remaining indices $i_{\sfa}$ and $j_{\hsfb}$ in the lower row by moving the latter upwards along the passing lines. Insert the resulting arrangement of indices from the upper row into the corresponding positions in the tensor component.
    \end{itemize}
    For example, the diagram $b_{1}$ in \eqref{eq:wB_example_diagrams} transform the components of a tensor in $V^{4,3}$ as follows:
    \begin{equation}
        b_{1} \,:\, t^{i_1 i_2 i_3 i_4}{}_{j_1 j_2 j_3} \mapsto \delta^{i_1}_{j_3} \delta^{i_4}_{j_1} \,t^{i_3 k i_2 l}{}_{k l j_2}.
    \end{equation}
    
    As a direct consequence of \cite[Lemma 1.2]{Koike_89}, in combination with Proposition \bref{prop:wB_generate}, the homomorphism \eqref{eq:hom_Br} is surjective. It is also injective if $N \geqslant m + n$ \cite[Theorem 5.8]{BCHLLS}. The author is unaware of a proof of the `only if' part in the literature, so we also prove that the latter condition is necessary for \eqref{eq:hom_Br} to be injective by analysing the traceless subspace (see Appendix \bref{app:proof_wB_hom} for proof).
    \begin{theorem}\label{thm:Brauer_homomorphism}
        For all $m,n \geqslant 1$ and $N\geqslant 1$ the homomorphism \eqref{eq:hom_Br} is surjective. It is also injective if and only if $N \geqslant m+n$. 
    \end{theorem}
    Describing the kernel of \eqref{eq:hom_Br} constitutes an interesting separated problem. The analogous problem for Brauer algebras was solved in \cite{Gavarini_LitRich}.

    \paragraph{Subalgebra generated by permutations and the complementary ideal.} 

    For any $\delta\in \mathbb{C}$, any diagram in $B_{m,n}(\delta)$ with no line crossing the wall is a permutation of the $m$ leftmost and $n$ rightmost nodes, and thus is an element of $\Sn{m}\times \Sn{n}$. Clearly the product of two permutations is again a permutation, so one identifies the corresponding subalgebra
    \begin{equation}\label{eq:Sn_in_wB}
        \mathbb{C}[\Sn{m}\times\Sn{n}] \subset B_{m,n}(\delta).
    \end{equation}
    
    In a complementary way, let $J \subset B_{m,n}(\delta)$ be the subspace spanned by diagrams with at least one arc, so that
    \begin{equation}\label{eq:wB_direct_sum_SJ}
        B_{m,n}(\delta) \cong \mathbb{C}[\Sn{m}\times\Sn{n}] \oplus J\quad\text{(isomorphism of vector spaces).} 
    \end{equation}
    Given a product of two diagrams (for example, $b_1 b_2$ for the diagrams in \eqref{eq:wB_example_diagrams}), the arcs with endpoints in the lower row of the former diagram (of $b_{1}$) and in the upper row of the latter (of $b_{2}$) stay intact, so $J$ is a two-sided ideal in $B_{m,n}(N)$. As a matter of a routine exercise one has the following:
    \begin{equation}\label{eq:isomorphism_J}
        B_{m,n}(\delta)\slash J \cong \mathbb{C}[\Sn{m}\times\Sn{n}]\quad \text{(isomorphsm of algebras).}
    \end{equation}

    \paragraph{Splitting idempotent in $B_{m,n}(\delta)$.} 
    In what follows we show that the direct-sum decomposition \eqref{eq:wB_direct_sum_SJ} also makes sense at the level of algebras when $B_{m,n}(\delta)$ is semisimple, and construct the corresponding central idempotent. The criterion of semisimplicity of $B_{m,n}(\delta)$ was given in \cite[Theorem 6.3]{CDDM}, and is summarised as follows.
    \begin{theorem}\label{thm:wB_semisimple}
        Let $m,n \geqslant 1$ and $\delta\in \mathbb{C}$, the algebra $B_{m,n}(\delta)$ is semisimple if and only if one of the following conditions holds:
            \begin{itemize}
                \item[1)] $\delta\notin \mathbb{Z}$, or otherwise
                \item[2)] $\delta\in \mathbb{Z}$ such that $|\delta| \geqslant m + n - 1$, or otherwise
                \item[3)] $\delta = 0$ and $(m,n) = (1,2), (1,3), (2,1)$ or $(3,1)$.
            \end{itemize}
    \end{theorem}
    
    In the case when the walled Brauer algebra is semisimple the following exact sequence of algebras splits:
    \begin{equation}\label{eq:exact_sequence_B}
        0 \xrightarrow{\phantom{m}} J \xrightarrow{\phantom{m}} B_{m,n}(\delta) \xrightarrow{\phantom{m}} \mathbb{C}\left[\Sn{m}\times\Sn{n}\right] \xrightarrow{\phantom{m}} 0,
    \end{equation}
    where one recalls the isomorphism \eqref{eq:isomorphism_J}. The corresponding splitting idempotent is constructed in a similar fashion as the traceless projector \eqref{eq:traceless_projector}. For all $r\in \{0,1,\ldots,\min(m,n)\}$ define the set of pairs of partitions $\Lambda^{(r)}_{m,n}$ as in \eqref{eq:index_modules_GL} by omitting the restriction  $\ell(\mu) + \ell(\nu) \leqslant N$, and set
    \begin{equation}\label{eq:index_reps_wB}
        \displaystyle\Lambda_{m,n} = \bigcup_{r = 0}^{\min(m,n)} \Lambda^{(r)}_{m,n}.
    \end{equation}
    Next, along similar lines as in \eqref{eq:operator_A_image} define the analogue of $\scrA_{m,n}$:
    \begin{equation}\label{eq:wB_A}
        A_{m,n} = \sum_{1 \leqslant \sfa \leqslant m, \, 1^{\prime} \leqslant \hsfb \leqslant n^{\prime}} t_{\sfa\hsfb}.
    \end{equation}
    The following assertion is the analogue of Theorem \bref{thm:eigenvalues} (see Appendix \bref{app:proof_wB_spec_A} for proof). 
    \begin{proposition}\label{prop:wB_spec_A}
        For any $m,n \geqslant 1$ and $\delta\in \mathbb{C}$ as in Theorem \ref{thm:wB_semisimple}, the left regular action of $A_{m,n}$ in $B_{m,n}(\delta)$ is diagonalisable. The elements $a\in\mathrm{spec}(A_{m,n})$ are of the form
        \begin{equation}\label{eq:eigenvalues_A_wB}
            a = r \delta + c(\rho\slash\mu) + c(\sigma\slash\nu)
        \end{equation}
        for all $r \in \{0,1,\ldots,\min(m,n)\}$, $(\mu,\nu) \in \Lambda^{(r)}_{m,n}$ and $(\rho,\sigma)\in \mathcal{P}_{m,n}$ such that $c^{\rho}_{\mu\beta}c^{\sigma}_{\nu\beta} \neq 0$ for some $\beta\subseteq \rho\cap\sigma$. In particular, $a = 0$ only for $r = 0$.
    \end{proposition}
    \vskip 2pt

    As a consequence, for the element
    \begin{equation}\label{eq:wB_splitting_idempotent}
        P_{m,n} = \prod_{a \in \mathrm{spec}(A_{m,n})^{\times}} \left(1 - \dfrac{1}{a} A_{m,n}\right) \in B_{m,n}(\delta)
    \end{equation}
    one has the following assertion.
    \begin{theorem}\label{thm:wB_splitting_idempotent}
        Let $m,n \geqslant 1$ and $\delta\in \mathbb{C}$ be as in Theorem \ref{thm:wB_semisimple}. Then the element \eqref{eq:wB_splitting_idempotent} is the splitting idempotent of the short exact sequence of algebras \eqref{eq:exact_sequence_B}. In particular,
        \begin{itemize}
            \item[1)] $P_{m,n} \in Z\big(B_{m,n}(\delta)\big)$, and
            \item[2)] $P_{m,n} B_{m,n}(\delta) \cong \mathbb{C}[\Sn{m}\times \Sn{n}]$ (isomorphism of algebras).
        \end{itemize}
    \end{theorem}

    \begin{proof}
        By recalling the isomorphism of vector spaces \eqref{eq:isomorphism_J} one is left to prove that the same isomorphism holds at the level of algebras. First, let us prove that \eqref{eq:wB_splitting_idempotent} annihilates $J$ from both sides and commutes with $\mathbb{C}[\Sn{m}\times\Sn{n}]$, which imples that $P_{m,n}\in Z\big(B_{m,n}(\delta)\big)$.
        \vskip 2pt
        
        One has $P_{m,n}J = 0$ by design of \eqref{eq:wB_splitting_idempotent}. Indeed, by \cite[Proposition 3.1]{CDDM} any diagram with $r$ arcs corresponds to a sum of elements of simple $B_{m,n}(\delta)$-modules labelled by $\Lambda^{(k)}_{m,n}$ with $k \geqslant r$, which are annihilated by $P_{m,n}$ unless $k = 0$. As a consequence, $JP_{m,n} = 0$ follows by applying the involutive anti-automorphism of $B_{m,n}$ which consists in reflecting walled diagrams with respect to the horizontal middle line. This anti-automorphism stabilises the ideal $J$ and preserves the element $A_{m,n}$ (and thus $P_{m,n}$). As a result, from the factorised form of \eqref{eq:wB_splitting_idempotent} one concludes that $P_{m,n}P_{m,n} = P_{m,n}$.
        \vskip 2pt

        To see that $P_{m,n}$ commutes with $\mathbb{C}[\Sn{m}\times\Sn{n}]$ it suffices to note that
        \begin{equation}
            A_{m,n} = \sum_{s\in \Sn{m}\times \Sn{n}} s t_{1,1^{\prime}} s^{-1}.
        \end{equation}
        
        Finally, we show that $P_{m,n}B_{m,n}(\delta) \cong B_{m,n}(\delta)\slash J$, so that the assertion follows by recalling \eqref{eq:isomorphism_J}. Consider the following homomorphism of algebras
        \begin{equation}
        \def\arraystretch{1.4}
        \begin{array}{rccc}
            \phi \,: & B_{m,n}(\delta)\slash J & \to & P_{m,n}B_{m,n}(\delta) \\
            \hfill & [x] & \mapsto & P_{m,n}x
        \end{array}  
        \end{equation}
        Since $P_{m,n}$ annihilates $J$ from both sides the above map is well-defined. Clearly, the above homomorphism is surjective. To prove that it is also injective note that for any $y \in \mathbb{C}[\Sn{m}\times\Sn{n}]$ one has $P_{m,n}y - y \in J$ (expand \eqref{eq:wB_splitting_idempotent} and recall that $J$ is an ideal). By recalling \eqref{eq:wB_direct_sum_SJ}, $P_{m,n}(x_{1} - x_{2}) = 0$ implies $x_{1} - x_{2} \in J$.
    \end{proof}

\paragraph{Relation to the traceless projector.} When the algebra $B_{m,n}(\delta)$ is semisimple, any finite-dimensional module over $B_{m,n}(\delta)$ decomposes as a direct sum of simple modules. With a slight abuse of notation, for any $(\mu,\nu) \in \Lambda_{m,n}$ denote $M^{(\mu,\nu)}_{m,n}$ the corresponding simple $B_{m,n}(\delta)$-module. The following assertion is reminiscent of Theorem \bref{thm:tensor_decomposition}, and relates particular representations of the walled Brauer algebra to subspaces $V_{0}^{m,n}$ (the proof goes along the same lines as for Proposition \bref{prop:wB_spec_A}).

\begin{proposition}\label{prop:projector_wB_to_C}
    Let $m,n \geqslant 1$ and $\delta\in \mathbb{C}$ be as in Theorem \ref{thm:wB_semisimple}, and let $\mathcal{M}$ be a $B_{m,n}(\delta)$-module.
    \begin{equation}
        \text{If}\quad \mathcal{M} \cong \bigoplus_{(\mu,\nu) \in \Lambda_{m,n}} \big(M^{(\mu,\nu)}_{m,n}\big)^{\oplus g_{\mu\nu}}\quad \text{then}\quad P_{m,n}\mathcal{M} \cong \bigoplus_{(\mu,\nu) \in \Lambda^{(0)}_{m,n}} \big(M^{(\mu,\nu)}_{m,n}\big)^{\oplus g_{\mu\nu}}.
    \end{equation}
    In particular, if $\delta = N \in \mathbb{Z}$ such that $N \geqslant m + n - 1$ one has
    \begin{equation}
        \mathfrak{b}(P_{m,n}) = \scrP_{m,n}.
    \end{equation}
\end{proposition}

\appendix
\section{Proofs}
\subsection{Proof of Lemma \ref{lem:operator_A}}\label{sec:proof_lemma_A}

        Consider the real form $U(N)$ of $GL(N)$ so that the space $V$ (as well as $V^{*}$) acquires a positive hermitian form (see Section \bref{sec:intro_applications}) which extends canonically to $V^{m,n}$. Given a basis $\{e_{i}\}$ of $V$ and the dual basis $\{e^{i}\}$ of $V^{*}$, one verifies that for any $T_{1},T_{2} \in V^{m,n}$ and for all $1\leqslant \sfa \leqslant m$ and $1^{\prime}\leqslant \hsfb \leqslant n^{\prime}$ one has
        \begin{equation}
            \big(T_{1},\tau_{\sfa\hsfb}(T_{2})\big) = \big(\tr{\sfa}{\hsfb}(T_{1}),\tr{\sfa}{\hsfb}(T_{2})\big),
        \end{equation}
        so $\tau_{\sfa\hsfb}$ is represented by a hermitian matrix. The same conclusion holds for $\scrA_{m,n}$ because
        \begin{equation}\label{eq:A_symmetric}
            \big(T_1,\scrA_{m,n}(T_2)\big) = \sum_{1\leqslant \sfa \leqslant m,\, 1^{\prime}\leqslant \hsfb \leqslant n^{\prime}} \big(\tr{\sfa}{\hsfb}(T_1),\tr{\sfa}{\hsfb}(T_2)\big),
        \end{equation}
        so the points (1) and (3) of the assertion follow.
        \vskip 2pt

        To prove the point (2) note that $V_{0}^{m,n} \subseteq \mathrm{ker}(\scrA_{m,n})$. For the opposite inclusion consider any $T\in V^{m,n}$ such that $\scrA_{m,n}(T) = 0$, and note that the right-hand-side of \eqref{eq:A_symmetric} with $T_{1} = T_{2} = T$ is zero only if vanishes each term thereof. The latter is the case only if $\tr{\sfa}{\hsfb}(T) = 0$ for all $1\leqslant \sfa \leqslant m$ and $1^{\prime}\leqslant \hsfb \leqslant n^{\prime}$.
        \vskip 2pt

        Finally, to prove that $\mathrm{Im}(\scrA_{m,n}) = V_{1}^{m,n}$ one readily notes that $\mathrm{Im}(\scrA_{m,n}) \subseteq V_{1}^{m,n}$. For the opposite inclusion note that $V_{1}^{m,n}$ is in the orthogonal complement of $V_{0}^{m,n}$ with respect to the hermitian form, recall \eqref{eq:intro_basis_dual}. Thus, $V_{1}^{m,n} \subseteq \mathrm{Im}(\scrA_{m,n})$ by the point (1).

\subsection{Proof of Lemma \ref{lem:estimate_multiplicities}}\label{sec:proof_lemma_estimate}

        The proof is carried out by means of the rational representations of $GL(N)$, recall \eqref{eq:branching_GLGL-GL}. For notations and details we refer the reader to \cite[pp. 65-67]{Koike_89}.
        \vskip 2pt
        
        Without loss of generality suppose $m \geqslant n$. Given $(\rho,\sigma) \in \mathcal{P}_{m,n}(N)$ one has the following decomposition formula for the product of irreducible characters of rational representations of $GL(N)$ \cite[Corollary 2.3.1]{Koike_89}:
        \begin{equation}\label{eq:decomposition_rational_Schur}
            s_{(\rho,\varnothing)}(x_{1},\ldots,x_{N})\cdot s_{(\varnothing,\sigma)}(x_{1},\ldots,x_{N}) = \sum_{\mu,\nu,\beta\in \mathcal{P}} c^{\rho}_{\mu\beta} c^{\sigma}_{\nu\beta}\, \tilde{\pi}_{N}\big([\mu,\nu]_{GL}\big)(x_{1},\ldots,x_{N}).
        \end{equation}
        Here $\tilde{\pi}_{N}$ is the specialisation homomorphism which maps universal characters $[\mu,\nu]_{GL}$ to irreducible characters of $GL(N)$ as described in \cite[Proposition 2.2]{Koike_89}. In particular, if $\ell(\mu) + \ell(\nu) \leqslant N$ then $\tilde{\pi}_{N}\big([\mu,\nu]_{GL}\big) = s_{(\mu,\nu)}$. Otherwise, when $\ell(\mu) + \ell(\nu) > N$, specialisation is determined via the rule given in \cite[p. 67]{Koike_89}.
        \vskip 2pt
        
        To prove that the inequality \eqref{eq:estimate_multiplicities} saturates for $N \geqslant m + n - 1$ let us show that in this case the multiplicity of $U^{(\mu,\nu)}$ in $U^{(\rho,\varnothing)}\otimes U^{(\varnothing,\sigma)}\big\downarrow{}_{GL(N)}$ equals the right-hand-side of \eqref{eq:estimate_multiplicities}. Note that the coefficient on the right-hand side of \eqref{eq:decomposition_rational_Schur} is non-zero only if $\mu\subseteq \rho$ and $\nu\subseteq \sigma$ \cite{Fulton_YT}. If $\ell(\rho) + \ell(\sigma) \leqslant N$, then one has $\ell(\mu) + \ell(\nu) \leqslant N$ for all non-zero coefficients on the right-hand side of \eqref{eq:decomposition_rational_Schur}. In the complementary case $\ell(\rho) + \ell(\sigma) > N$ one has $\ell(\mu) + \ell(\nu) > N$ only for $(\rho) = (\mu) = (1_{m})$ and $(\sigma) = (\nu) = (1_n)$ when $N = m + n - 1$. By applying the rule given in \cite[p. 67]{Koike_89} one has $\tilde{\pi}_{N}\big([\mu,\nu]_{GL}\big) = 0$, so \eqref{eq:decomposition_rational_Schur} reduces to
        \begin{equation}
            s_{(\rho,\varnothing)}(x_{1},\ldots,x_{N})\cdot s_{(\varnothing,\sigma)}(x_{1},\ldots,x_{N}) = \sum_{(\mu,\nu) \in \Lambda_{m,n}(N),\,\beta\in \mathcal{P}} c^{\rho}_{\mu\beta} c^{\sigma}_{\nu\beta}\, s_{(\mu,\nu)}(x_{1},\ldots,x_{N}).
        \end{equation} 
        \vskip 2pt

        The case $n = 1$, $m \geqslant 2$ and $N \leqslant m - 1$ follows along the same lines: one has $\ell(\mu) + \ell(\nu) > N$ only if $\ell(\mu) = N$ and $\sigma = (1)$, in which case the rule given in \cite[p. 67]{Koike_89} gives $\tilde{\pi}_{N}\big([\mu,\nu]_{GL}\big) = 0$. The case $m = 1$, $n \geqslant 2$ and $N \leqslant n - 1$ follows by interchanging $V$ and $V^{*}$.
        \vskip 2pt

        Finally, for $N = 1$ and $m,n\geqslant 2$ one has the only possibility $\rho = (m)$ and $\sigma = (n)$. For any $(\mu,\nu) \in \Lambda_{m,n}(1)$ among the two partitions at least one is necessarily empty, so one has $\mu = (m - r)$, $\nu = \varnothing$ or $\mu = \varnothing$, $\nu = (n - r)$ for some non-negative integer $r$. Without loss of generality let $m\geqslant n$, so the right-hand side of \eqref{eq:estimate_multiplicities} is non-zero only if $\mu = (m-n)$, $\nu = \varnothing$. In this case $\beta = (n)$, so the right-hand side of \eqref{eq:estimate_multiplicities} equals to $1$. To check that $c^{\mu\nu}_{\rho\sigma}(1) = 1$ recall \eqref{eq:multiplicity_LR_coef} and note that $c^{\lambda}_{\rho\bar{\sigma}} \neq 0$ implies the only possibility $\lambda = (m)$, in which case $\mathfrak{s}^{-1}[(m),n] = (\mu,\nu)$.
        \vskip 2pt

        For the resting cases $m,n\geqslant 2$ and $2 \leqslant N \leqslant m+n-2$ let us show that one can find pairs of partitions $(\rho,\sigma) \in \mathcal{P}_{m,n}(N)$ and $(\mu,\nu)\in \Lambda_{m,n}(N)$ such that $c^{\rho}_{\mu\beta}c^{\sigma}_{\nu\beta} \neq 0$ for $\beta = (1)$ while $c^{\mu\nu}_{\rho\sigma}(N) = 0$. One has the following three cases where the latter condition boils down to $c^{\mu + \bar{\nu}}_{\rho\bar{\sigma}} = 0$ (recall \eqref{eq:map_staircase} and \eqref{eq:multiplicity_LR_coef}).
        \vskip 2pt
        \begin{itemize}
            \item[1)] Let $n \leqslant N$, then take $\rho = (m+n-N-1,1_{N-n+1})$, $\sigma = (1_{n})$ and $\mu = (m+n-N - 1,1_{N-n})$, $\nu = (1_{n-1})$. One has $\mu + \bar{\nu} = (m+n-N,2_{N-n})$, so $c^{\mu+\bar{\nu}}_{\rho,\bar{\sigma}} = 0$ since $\rho \not\subseteq \mu+\bar{\nu}$.
            \item[2)] Let $n = k(N-1) + 1$ for some integer $k \geqslant 1$. Then take $\rho = (m-1,1)$, $\sigma = (k_{N-1},1)$ and $\mu = (m-1)$, $\nu = (k_{N-1})$. One has $\mu + \bar{\nu} = (k + m - 1)$, so $c^{\mu+\bar{\nu}}_{\rho,\bar{\sigma}} = 0$ since $\rho \not\subseteq \mu+\bar{\nu}$.
            \item[3)] Let $n = k(N-1) + 1 + t$ for some non-negative integers $k \geqslant 1$ and $1\leqslant t < N-1$ (in this case $N \geqslant 3$). Take $\rho = (m-1,1)$, $\sigma = \big((k+1)_{t},k_{N-t-1},1\big)$ and $\mu = (m-1)$, $\nu = \big((k+1)_{t},k_{N-t-1}\big)$. One has $\mu + \bar{\nu} = (m + k,1_{N-t-1})$, $\bar{\sigma} = (k,1_{N-t-1})$, and thus $c^{\mu+\bar{\nu}}_{\rho,\bar{\sigma}} = 0$. Indeed, the only skew tableau of shape $(\mu+\bar{\nu})\slash \bar{\sigma}$ and of content $\rho$ is not a Littlewood-Richardson tableau \cite{Fulton_YT}.
        \end{itemize}

\subsection{Proof of Theorem \ref{thm:Brauer_homomorphism}}\label{app:proof_wB_hom}

        The homomorphism \eqref{eq:hom_Br} is surjective, and is also injective if $N\geqslant m+n$ \cite{BCHLLS}. To prove necessity of the latter condition, for $N < m+n$ it suffices to find one non-zero element of $B_{m,n}(N)$ which is mapped to $0\in C_{m,n}(N)$. Since \eqref{eq:hom_Br} is surjective consider $Q_{m,n} \in B_{m,n}(N)$ such that $\mathfrak{b}(Q_{m,n}) = \scrP_{m,n}$. By recalling \eqref{eq:traceless_projector_X}, one can take $Q_{m,n} = 1 + X_{m,n}$ with $X_{m,n}\in J$.
        \vskip 2pt
        
        Take $(\rho,\sigma) \in \mathcal{P}_{m,n}(N)$ such that $\ell(\rho) + \ell(\sigma) > N$, and consider the element $q = Q_{m,n}\, Z^{(\rho)}\otimes Z^{(\sigma)} \in B_{m,n}(N)$ (recall \eqref{eq:otimes_SnSn}). One has $\mathfrak{b}(q) = 0$ since $U^{(\rho,\sigma)} = \{0\}$ the traceless projection of $U^{(\rho,\varnothing)}\otimes U^{(\varnothing,\sigma)}$. To see that $q \neq 0$ recall that $J\subset B_{m,n}(N)$ is an ideal, that $Z^{(\rho)}\otimes Z^{(\sigma)} \in \mathbb{C}[\Sn{m}\times\Sn{n}]$ and that $J \cap \mathbb{C}[\Sn{m}\times\Sn{n}] = \{0\}$.

\subsection{Proof of Proposition \ref{prop:wB_spec_A}}\label{app:proof_wB_spec_A}
        Since the algebra $B_{m,n}(\delta)$ is semisimple by Theorem \bref{thm:wB_semisimple}, the left regular module decomposes as a direct sum of simple $B_{m,n}(\delta)$-modules. By \cite[Theorem 2.7]{CDDM} $B_{m,n}(\delta)$ is cellular, with the equivalence classes of simple $B_{m,n}(\delta)$-modules represented by cell modules $\Delta^{(\mu,\nu)}_{m,n}$ for all $(\mu,\nu) \in \Lambda_{m,n}$ (see \cite{CDDM} and references therein). By \cite[Theorem 6.1]{CDDM} the restriction of a cell module $\Delta^{(\mu,\nu)}_{m,n}$ with $(\mu,\nu) \in \Lambda^{(r)}_{m,n}$ to $\mathbb{C}[\Sn{m}\times \Sn{n}]$ decomposes as a direct sum of $L^{(\rho)}\otimes L^{(\sigma)}$ with all $(\rho,\sigma) \in \mathcal{P}_{m,n}$ such that $c^{\rho}_{\mu\beta}c^{\sigma}_{\nu\beta} \neq 0$ for some partition $\beta$. In this case one necessarily has $\beta \subseteq \rho\cap \sigma$ by the properties of the Littlewood-Richardson coefficients, see \cite{Fulton_YT}. Diagonalisability of $A_{m,n}$ together with the set of eigenvalues follows by virtue of \cite[Lemma 4.1]{CDDM}.
        \vskip 2pt

        It is left to prove that $r \neq 0$ implies $a \neq 0$. The latter implication is straightforward for $\delta \in \mathbb{C}\backslash \mathbb{Z}$, and is a matter of a simple check for $\delta = 0$. For any integer $\delta = N \geqslant 1$ the action of $B_{m,n}(N)$ in $V^{m,n}$ generates the centraliser algebra $C_{m,n}(N)$. By assuming $N \geqslant m + n - 1$, let us show that for any $r \in \{1,\ldots,\min(m,n)\}$ each element \eqref{eq:eigenvalues_A_wB} equals the element \eqref{eq:eigenvalue} with the same $(\mu,\nu) \in \Lambda^{(r)}_{m,n}(N)$ and $(\rho,\sigma)\in \mathcal{P}_{m,n}(N)$, so one has $a>0$ by Theorem \bref{thm:eigenvalues}. Indeed, note that in this case: {\it (i)} $\Lambda^{(r)}_{m,n}(N) = \Lambda^{(r)}_{m,n}$ because the only pair $(\mu,\nu)\in \Lambda_{m,n}$ such that $\ell(\mu) + \ell(\nu) > \ell$ is $\big((1_{m}),(1_{n})\big) \in \Lambda^{(0)}_{m,n}$, and {\it (ii)} $\mathcal{P}_{m,n}(N) = \mathcal{P}_{m,n}$ because $m,n\geqslant 1$ so $N \geqslant m$ and $N \geqslant n$. Finally, the condition  $\sum_{\beta} c^{\rho}_{\beta\mu} c^{\sigma}_{\beta\nu} \neq 0$ is equivalent to $c^{\mu\nu}_{\rho\sigma} \neq 0$ by Lemma \bref{lem:estimate_multiplicities}.
        \vskip 2pt

        The case of negative integers $\delta = N \leqslant -(m+n-1)$ follows from the previous case via passing to dual partitions.
        
\newpage


\end{document}